\documentclass[reqno]{asl} 

\usepackage{hyperref}

\title{Undecidability in Relevant Logic}


\author{Søren Brinck Knudstorp}
\revauthor{Knudstorp, Søren Brinck}

\address{ILLC \& Philosophy\\
University of Amsterdam\\
Amsterdam, the Netherlands}

\email{s.b.knudstorp@uva.nl}

\urladdr{https://knudstorp.github.io/}

\newtheorem{Theorem}{Theorem}[section]
\newtheorem{Proposition}[Theorem]{Proposition}
\newtheorem{Lemma}[Theorem]{Lemma}
\newtheorem{Corollary}[Theorem]{Corollary}

\theoremstyle{definition}
\newtheorem{Definition}[Theorem]{Definition}
\newtheorem{Fact}[Theorem]{Fact}

\newtheorem{Remark}[Theorem]{Remark}
\newtheorem{Example}[Theorem]{Example}

\renewcommand{\qedsymbol}{$\Box$} 

\usepackage{graphicx}

\usepackage{tikz} 
\usetikzlibrary{arrows.meta, positioning}
\usepackage{mathdots} 
\usepackage{enumitem} 
\usepackage{bussproofs} 
\usepackage{multicol} 
\usepackage{float} 

\begin{document}

\begin{abstract}
    We prove undecidability for every positive relevant logic extending the system axiomatized by hypothetical syllogism, prefixing, and suffixing and contained in the logic of the semilattice frame $(\mathcal{P}_{\mathrm{fin}}(\mathbb{N}), \cup, \varnothing)$.

    This settles the longstanding decision problem for the semilattice relevant logic $\mathsf{S}$ in the negative, contrary to prevailing expectations of decidability. 
    It also provides a new proof of Urquhart's undecidability theorem for $\mathsf{R}$, $\mathsf{E}$, and $\mathsf{T}$~\cite{urquhart84:jsl}, now by reduction from the Wang tiling problem for arbitrarily large finite isosceles right triangular regions of the plane.
%
\end{abstract}

\maketitle

\section{Introduction}\label{intro}

Initially presented proof-theoretically, the problem of providing a model theory for relevant logics led Urquhart~\cite{Urquhart72,Urquhart73} to develop semilattice semantics. In the signature $\{\to,\land\}$, the resulting semilattice logic $\mathsf{S}_{\to,\land}$ enjoys the finite model property, is decidable, and coincides with $\mathsf{R}_{\to,\land}$, the implication-conjunction fragment of the relevant logic $\mathsf{R}$. In the presence of disjunction, the logic $ \mathsf{S}=\mathsf{S}_{\to,\land,\lor}$ properly extends its $\mathsf{R}$-counterpart $\mathsf{R}^+=\mathsf{R}_{\to,\land,\lor}$. While $\mathsf{S}$ is known to be recursively enumerable and was axiomatized by Fine~\cite{Fine76} and Charlwood~\cite{Charlwood81}, it has been an enduring open problem whether it is decidable. 

More broadly, decision problems for relevant logics have been studied since the beginnings of the field, with an early mention in Curry and Craig's review~\cite{CurryCraig1953} of Church's weak theory of implication~\cite{Church1951}, which corresponds to the implication fragment $\mathsf{R}_\to$ of $\mathsf{R}$. Kripke~\cite{Kripke1959Entailment} proved $\mathsf{R}_\to$ decidable, and Meyer~\cite{Meyer1966} extended the result to the implication-conjunction fragment $\mathsf{R}_{\to, \land}$. 
Still, the decision problem for $\mathsf{R}^+$ (let alone for $\mathsf{R}$) remained open for another two decades, until 
 Urquhart~\cite{urquhart84:jsl} settled the matter. Using a Lipshitz--Hutchinson-style~\cite{Lipshitz74, Hutchinson1973} reduction that passes via a coordinatization theorem of von Neumann~\cite{vonNeumann1960}, he showed that many relevant logics, including $\mathsf{R}^+$, $\mathsf{E}^+$ and $\mathsf{T}^+$, are undecidable. Intriguingly, as noted in~\cite{urquhart84:jsl}, 
$\mathsf{S}$ eluded these techniques, and its decision problem remained unresolved. Eventually, this led experts, including Urquhart~\cite{Urquhart16}, to conjecture that $\mathsf{S}$ is decidable.
Contrary to this conjecture, we show that $\mathsf{S}$ is undecidable. 
In doing so, we prove undecidability for a broad class of relevant logics, among which are $\mathsf{R}^+$, $\mathsf{E}^+$ and $\mathsf{T}^+$ (to the author's knowledge, the proof of~\cite{urquhart84:jsl} remains the only prior undecidability proof for these systems).

The weakest system we prove undecidable is the relevant logic $\mathsf{TWJ}^+$. It extends the logic $\mathsf{TW}^+$ (which, in contrast, is decidable~\cite{Giambrone1985}) by the axiom for hypothetical syllogism
\[
    [(\varphi\to\psi)\land (\psi\to\chi)]\to (\varphi\to\chi).
\]
$\mathsf{TW}^+$, in turn, extends the minimal basic relevant logic $\mathsf{B}^+$ by the axioms for prefixing and suffixing
\[
    (\varphi\to\psi)\to[(\chi\to\varphi)\to(\chi\to\psi)]\quad \text{and}\quad (\varphi\to\psi)\to [(\psi\to\chi)\to(\varphi\to\chi)].
\]
At a more conceptual level, our aim is to better understand the boundary between the solvable and the unsolvable. To this end, our proof identifies these axioms, in combination, as a source of undecidability and explains the mechanism by which they lead to it.

The proof is by reduction from the Wang tiling problem, introduced by Wang~\cite{Wang1961,Wang1963} in the study of the decidability of the $\forall\exists\forall$-fragment of first-order logic and proven undecidable by Berger~\cite{Berger}. A Wang tile is a square with coloured sides and a fixed orientation (no rotation or reflection allowed). Given a finite set of tiles $\mathcal{W}$, the problem asks whether the plane can be covered with translated copies of tiles from $\mathcal{W}$ placed side-by-side so that adjacent tiles have matching colours on shared sides. 

To code the tiling problem, for each finite tile set $\mathcal{W}$, we construct a formula $\psi^{}_{\mathcal{W}}$ and prove that 
\begin{enumerate}[label=(\roman*)]
	\item if $\mathsf{TWJ}^+\nvdash \psi^{}_{\mathcal{W}}$, then $\mathcal{W}$ tiles the plane, and
	\item if $\mathcal{W}$ tiles the plane, then $\mathsf{S}\nvdash \psi^{}_{\mathcal{W}}$.
\end{enumerate}
We shall in fact prove a stronger statement than (ii), but (ii) as stated suffices to conclude that every logic $\mathsf{L}$ with $\mathsf{TWJ}^+\subseteq\mathsf{L}\subseteq \mathsf{S}$ is undecidable. It is worth noting that we do not prove (i) directly. Instead, we show that if $\mathsf{TWJ}^+\nvdash \psi^{}_{\mathcal{W}}$, then $\mathcal{W}$ tiles isosceles right triangular regions of the plane of arbitrarily large finite size, which, by a Kőnig's Lemma argument, implies (i).

An earlier version of this work appeared as the conference paper~\cite{Knudstorp2024}; the present journal version simplifies the proof, generalizes the result, and explains the tiling intuitions in greater depth. The latter is also relevant in view of the apparent breadth of the methods developed here and elsewhere, including in~\cite{Knudstorp2025a, GalatosJipsenKnudstorpRevantha:BIund} where related tiling-based techniques are applied in modal logic and bunched implication logic.

The agenda of the paper is as follows. \S\ref{sec:pre} contains the preliminaries, defining the relevant logics of concern as well as the tiling problem. 
\S\ref{sec:TWJandtiling} clarifies why the axioms of $\mathsf{TWJ}^+$ are a source of undecidability by relating them to the ability to encode tilings of arbitrarily large finite isosceles right triangular regions. \S\ref{sec:tilingformulas} defines the tiling formulas. \S\ref{sec:tilinglemmas} proves the tiling lemmas stated in (i) and (ii) above, and \S\ref{sec:results} collects the consequences.

\section{Preliminaries}\label{sec:pre} In this section, we first present relevant logics axiomatically, then give their Routley--Meyer frame semantics -- which we will use to encode tilings -- and finally define the tiling problem used in the reduction.

For further background on relevant logic, see~\cite{AndersonBelnap75, AndersonBelnapDunn92, DunnRestall02,Standefer26}.

\subsection{Syntax}\label{subsec:syntax}

\begin{Definition}
    The formulas $\varphi\in \mathcal{L}$ of the language $\mathcal{L}$ are given by the BNF-grammar
        \begin{align*}
        \varphi\mathrel{::=} p \mid\varphi\land\varphi\mid\varphi\lor\varphi\mid\varphi\to\varphi,
        \end{align*}
    where $p\in \mathsf{Prop}$ for $\mathsf{Prop}$ a denumerable set of propositional letters. We follow the convention that $\land, \lor$ bind tighter than $\to$.
    
    While this will be our language of interest, as it suffices for undecidability, let us note that in the so-called positive relevant language \textit{fusion} $\varphi\circ \psi$ and the \textit{Ackermann truth constant} $\mathfrak{t}$ are sometimes included.
\end{Definition}

The weakest system for which we will prove undecidability is the logic $\mathsf{TWJ}^+$. 
It can be given a Hilbert-style axiomatization with the following axiom schemes:
\begin{enumerate}
    \item $\varphi\to\varphi$
    \item $\varphi\land \psi\to \varphi$
    \item $\varphi\land \psi\to \psi$
    \item $(\varphi\to\psi)\land (\varphi\to \chi)\to (\varphi\to \psi\land\chi)$
    \item $\varphi\to\varphi\lor \psi$ 
    \item $\psi\to\varphi\lor \psi$
    \item $(\varphi\to\chi)\land (\psi\to\chi)\to (\varphi\lor\psi\to\chi)$
    \item $\varphi\land (\psi\lor\chi)\to (\varphi\land \psi)\lor(\varphi\land \chi)$
    \item $(\varphi\to\psi)\land (\psi\to\chi)\to (\varphi\to\chi)$ \hfill (Hypothetical syllogism)
    \item $(\varphi\to\psi)\to[(\chi\to\varphi)\to(\chi\to\psi)]$ \hfill (Prefixing)
    \item $(\varphi\to\psi)\to [(\psi\to\chi)\to(\varphi\to\chi)]$ \hfill (Suffixing)
\end{enumerate}
And with the rules of modus ponens and adjunction:
\begin{center}

\begin{minipage}{0.45\textwidth}
\begin{prooftree}
\AxiomC{$\varphi$}
\AxiomC{$\varphi \to \psi$}
\RightLabel{\scriptsize modus ponens}
\BinaryInfC{$\psi$}
\end{prooftree}
\end{minipage}
\begin{minipage}{0.45\textwidth}
\begin{prooftree}
\AxiomC{$\varphi$}
\AxiomC{$\psi$}
\RightLabel{\scriptsize adjunction}
\BinaryInfC{$\varphi \land \psi$}
\end{prooftree}
\end{minipage}

\end{center}
The last three axioms, 9-11, were named not for general prominence, but for their essential role in our undecidability argument. Replacing them with their corresponding rules yields the logic $\mathsf{B}^+$, the $\mathcal{L}$-fragment of the relevant logic $\mathsf{B}$.\footnote{To be precise, the rule of hypothetical syllogism is not needed, as it is derivable from axiom 2 and 3 together with modus ponens and the suffixing rule.} 

The relevant logics $\mathsf{C}^+, \mathsf{T}^+, \mathsf{E}^+,$ and $\mathsf{R}^+$ are obtained by cumulatively adding the following axiom schemes:

\begin{enumerate}
    \item[13.] $\varphi\land(\varphi\to\psi)\to\psi$  \hfill (Pseudo-modus ponens)
    \item[14.] $(\varphi\to(\varphi\to\psi))\to(\varphi\to\psi)$  \hfill (Contraction)
    \item[15.] $((\varphi\to\varphi)\to \psi)\to \psi$ and $\Box\varphi\land\Box\psi\to\Box(\varphi\land \psi)$  \hfill (E axioms)\footnote{Here, $\Box$ is an abbreviated operator: $\Box \alpha\mathrel{:=}(\alpha\to\alpha)\to\alpha$. Written without abbreviations, the latter scheme is: $[((\varphi\to\varphi)\to\varphi)\land ((\psi\to\psi)\to\psi)]\to [(\varphi\land \psi\to\varphi\land \psi)\to\varphi\land \psi]$. In general, language reducts of the logic $\mathsf{E}$ can be more perspicuously presented when additional vocabulary is available; for example, in a language with Ackermann's constant $\mathfrak{t}$, this axiom-scheme pair can be replaced with the simpler E axiom $(\mathfrak{t}\to \varphi)\to\varphi$.}
    \item[16.] $\varphi\to((\varphi\to\psi)\to\psi)$ \hfill (Assertion)
\end{enumerate}
That is, 
\begin{alignat*}{2}
    \mathsf{C}^+&=\mathsf{TWJ}^+\oplus \text{Pseudo-modus ponens}, &\qquad \mathsf{T}^+&=\mathsf{C}^+\oplus \text{Contraction},\\
    \mathsf{E}^+&=\mathsf{T}^+\oplus \text{E axioms}, &\qquad \mathsf{R}^+&=\mathsf{E}^+\oplus \text{Assertion},
\end{alignat*}
where we, given a logic $\mathsf{L}$ and a scheme or set of schemes $\Phi$, by the notation `$\mathsf{L}\oplus \Phi$' denote the logic obtained by adding $\Phi$ to $\mathsf{L}$ and closing under the rules of modus ponens and adjunction. 

This presentation is not chosen for minimality (e.g., $\mathsf{R}^+=\mathsf{B}^+\oplus \text{Contraction}\oplus \text{Suffixing}\oplus \text{Assertion}$), but to clarify their structure as a chain of increasingly stronger logics, as we shall further substantiate after introducing their frame semantics.

\subsection{Semantics}
Formulas of the language $\mathcal{L}$ can be interpreted on \textit{Routley--Meyer frames}, defined as follows.

\begin{Definition}\label{def:basicframes}
    A \textit{(Routley--Meyer) frame} for $\mathcal{L}$ is a triple $\mathfrak{F}=(K, R, N)$ where $R\subseteq K\times K\times K$ is a ternary relation, $N\subseteq K$ is a set whose elements are denoted \textit{normal points}, and which is subject to the following conditions, where we write $x\leq y$ for $\exists n\in N(Rnxy)$:
    \begin{itemize}
    \item \makebox[2.7cm][l]{(reflexivity)} $x\leq x$,
    \item \makebox[2.7cm][l]{(transitivity)} if $x\leq y\leq z$, then $x\leq z$,
    \item \makebox[2.7cm][l]{(upset)} if $n\in N$ and $n\leq m$, then $m\in N$,
    \item \makebox[2.7cm][l]{(down-down-up)} 
    if $Rxyz$, $x^-\leq x, y^-\leq y$, and $z\leq z^+$, then $Rx^{-}y^{-}z^{+}$.
\end{itemize}
\end{Definition}

\begin{Remark}\label{rm:operationalperspective}
    From a ternary relation $R\subseteq K\times K\times K$, we can define a binary operation $\cdot_R:K\times K\to \mathcal{P}(K)$ given by
        \[
            x\cdot_R y\mathrel{:=}\{z\mid Rxyz\},
        \]
    which lifts to sets $X,Y\subseteq K$ by defining
    \begin{align*}
        X\cdot_R Y\mathrel{:=}\bigcup_{x\in X, y\in Y}x\cdot_R y, && X\cdot_R y\mathrel{:=}X\cdot_R\{y\},&& x\cdot_R Y\mathrel{:=}\{x\}\cdot_R Y.
    \end{align*}
    Conversely, from a binary operation $\cdot:K\times K\to \mathcal{P}(K)$, we can define a ternary relation $R\mathrel{:=}\{(x,y,z)\mid x\cdot y\ni z\}$. (When we now and later write `$x\cdot y\ni z$', instead of `$z\in x\cdot y$', it is merely a notational choice to aid in parsing it as $Rxyz$, not as $Rzxy$.)
    
    Clearly, these transformations are inverses, hence we may as well define frames as triples $\mathfrak{F}=(K, \cdot, N)$ where $\cdot:K\times K\to \mathcal{P}(K)$ is a binary operation, $N\subseteq K$, and the conditions of Definition~\ref{def:basicframes} are met when we write $x\leq y$ for $N\cdot x\ni y $. 
\end{Remark}
For the purposes of this paper, it will be convenient to work with the binary-operation definition of frames, and we proceed accordingly. Given a frame $\mathfrak{F}=(K, \cdot, N)$, we write $\mathcal{U}(K)$ for the set of $\leq$-\textit{upsets}; that is, subsets $X\subseteq K$ such that  $x\in X$ and $x\leq y$ imply $y\in X$.

\begin{Definition}
    A \textit{model} is a pair $(\mathfrak{F}, V)$ where $\mathfrak{F}=(K, \cdot, N)$ is a frame and $V:\mathsf{Prop}\to \mathcal{U}(K)$ is a function, called a \textit{valuation}. For a model $\mathfrak{M}=(\mathfrak{F}, V)$, we say that $\mathfrak{M}$ is a model \textit{over} the frame $\mathfrak{F}$.
\end{Definition}

\begin{Definition}\label{def:clauses}
    Given a model $\mathfrak{M}=(K, \cdot, N, V)$, a point $x\in K$, and a formula $\varphi\in \mathcal{L}$, we say that $x$ \textit{satisfies} $\varphi$ and write $\mathfrak{M}, x\Vdash \varphi$, or simply $x\Vdash\varphi$, according to the following clauses:
    \begin{align*}
        &\mathfrak{M}, x\Vdash p && \text{iff} && x\in V(p),\\
        &\mathfrak{M}, x\Vdash \varphi\land\psi && \text{iff} && \mathfrak{M}, x\Vdash \varphi \text{ and } \mathfrak{M}, x\Vdash \psi,\\
        &\mathfrak{M}, x\Vdash \varphi\lor\psi && \text{iff} && \mathfrak{M}, x\Vdash \varphi \text{ or } \mathfrak{M}, x\Vdash \psi,\\
        &\mathfrak{M}, x\Vdash \varphi\to\psi && \text{iff} && \text{for all $y,z\in K$: if } \mathfrak{M}, y\Vdash \varphi\text{ and $x\cdot y\ni z$, then } \mathfrak{M}, z\Vdash \psi.
    \end{align*}
    Although only these clauses are needed for the undecidability results, we include the clauses for fusion $\circ$ and the Ackermann constant $\mathfrak{t}$ for completeness:
    \begin{align*}
        &\mathfrak{M}, x\Vdash \mathfrak{t} && \text{iff} && x\in N,\\
        &\mathfrak{M}, x\Vdash \varphi\circ\psi && \text{iff} && \text{there exist $y,z\in K$ s.t. } \mathfrak{M}, y\Vdash \varphi\text{, $\mathfrak{M}, z\Vdash \psi $, and $y\cdot z\ni x$.}
    \end{align*}
    
    A formula $\varphi\in\mathcal{L}$ is \textit{valid in a frame} $\mathfrak{F}=(K, \cdot, N)$, written $\mathfrak{F}\vDash\varphi$, if for every model $\mathfrak{M}=(\mathfrak{F},V)$ over $\mathfrak{F}$ and every normal point $n\in N$, $\mathfrak{M},n\Vdash \varphi$. It is \textit{valid in a class of frames} $\mathcal{C}$, written $\mathcal{C}\vDash\varphi$, if it is valid in every frame $\mathfrak{F}\in\mathcal{C}$. It is \textit{refuted by a frame} $\mathfrak{F}$, written $\mathfrak{F}\nvDash \varphi$, if it is not valid in $\mathfrak{F}$. Relative to a model $\mathfrak{M}=(K, \cdot, N, V)$, we also say that a point $x\in K$ \textit{refutes} $\varphi$ if $\mathfrak{M}, x\nVdash \varphi$, although we only say that $\mathfrak{M}$ itself \textit{refutes} $\varphi$ if there is a normal point $n\in N$ refuting $\varphi$.
\end{Definition}
\begin{Remark}\label{rm:verification}
     As valuations $V:\mathsf{Prop}\to \mathcal{U}(K)$ take only upsets as values, propositional letters are persistent: if $x\Vdash p$ and $x\leq y$, then $y\Vdash p$. An induction readily extends this to arbitrary formulas $\varphi\in\mathcal{L}$: if $x\Vdash \varphi$ and $x\leq y$, then $y\Vdash \varphi$. Persistence along with reflexivity of $\leq$, in turn, imply that for any model $\mathfrak{M}=(K, R, N, V)$ and implication $\varphi\to\psi\in\mathcal{L}$, $\mathfrak{M}$ refutes $\varphi\to\psi$ iff there is some $x\in K$ such that $x\Vdash \varphi$ but $x\nVdash \psi$.
\end{Remark}

Given a class of frames $\mathcal{C}$, we define its logic $\mathrm{Log}(\mathcal{C})$ as the set of formulas valid in $\mathcal{C}$, i.e., $\mathrm{Log}(\mathcal{C})\mathrel{:=}\{\varphi \in\mathcal{L\mid \mathcal{C}\vDash \varphi}\}$. It is a standard completeness result that the logic of the class of all frames is the relevant logic $\mathsf{B}^+$. Likewise is it a standard result that complete semantics for $\mathsf{TWJ}^+$, $\mathsf{C}^+$, $\mathsf{T}^+$, $\mathsf{E}^+$ and $\mathsf{R}^+$ can be obtained through the following correspondences (as the formulas are canonical).

\begin{Proposition}[Correspondence]\label{pr:correspondence}
    Let $\mathfrak{F}=(K,\cdot, N)$ be a frame. Then the following correspondences hold:
    \begin{alignat*}{3}
        \text{\textnormal{Hypothetical syllogism} is valid on $\mathfrak{F}$} &\qquad \text{iff} &\qquad x\cdot y&\subseteq x\cdot (x\cdot y),\\
        \text{\textnormal{Prefixing} is valid on $\mathfrak{F}$} &\qquad \text{iff} &\qquad (x\cdot y)\cdot z&\subseteq x\cdot (y\cdot z),\\
        \text{\textnormal{Suffixing} is valid on $\mathfrak{F}$} &\qquad \text{iff} &\qquad (x\cdot y)\cdot z&\subseteq y\cdot (x\cdot z),\\
        \text{\textnormal{Pseudo-modus ponens} is valid on $\mathfrak{F}$} &\qquad \text{iff} &\qquad x&\in x\cdot x,\\
        \text{\textnormal{Contraction} is valid on $\mathfrak{F}$} &\qquad \text{iff} &\qquad x\cdot y&\subseteq (x\cdot y)\cdot y,\\
        \text{\textnormal{E axiom} is valid on $\mathfrak{F}$} &\qquad \text{iff} &\qquad x&\in x\cdot N,\\
        \text{\textnormal{Assertion} is valid on $\mathfrak{F}$} &\qquad \text{iff} &\qquad x\cdot y&\subseteq y\cdot x.    
    \end{alignat*}
\end{Proposition}
One may therefore verify the strict inclusions
\[
    \mathsf{TWJ}^+\quad \subsetneq \quad \mathsf{C}^+\quad \subsetneq\quad  \mathsf{T}^+\quad \subsetneq\quad \mathsf{E}^+\quad \subsetneq \quad \mathsf{R}^+.\footnotemark
\]
For\footnotetext{See, e.g.,~\cite{Standefer26} for completeness-via-canonicity proofs.\\
\cite{Standefer26} doesn't include explicit mention of the first inclusion, nor have I found it stated elsewhere, so we prove it here. First, note that, modulo $\mathsf{B}^+\oplus\textnormal{Prefixing}$, Pseudo-modus ponens 
implies Hypothetical syllogism, because then
    \[
        x\cdot y\subseteq (x\cdot x)\cdot y\subseteq x\cdot (x\cdot y).
    \]
In contrast, a frame satisfying the first three frame conditions need not satisfy the fourth, as witnessed by the frame $(\{n,x\}, \cdot, \{n\})$ where $n\cdot n=\{n\}$, $n\cdot x=\{x\}$, and otherwise $y\cdot z=\varnothing$. Consequently, $\mathsf{TWJ}^+ \subsetneq \mathsf{C}^+$. 
} our purposes, we highlight that $\mathsf{TWJ}^+$ is sound and complete for the frames $\mathfrak{F}=(K, \cdot, N)$ satisfying
\begin{alignat*}{2}
        &\text{(h)}  &\qquad x\cdot y&\subseteq x\cdot (x\cdot y), \\
        &\text{(p)}  &\qquad (x\cdot y)\cdot z&\subseteq x\cdot (y\cdot z),\\
        &\text{(s)}  &\qquad (x\cdot y)\cdot z&\subseteq y\cdot (x\cdot z).    
    \end{alignat*}
These three frame conditions -- (h), (p) and (s) -- will be instrumental in encoding the tiling problem: in Lemma~\ref{lm:tiling}, we will show that if an $\mathsf{TWJ}^+$-frame (i.e., a frame satisfying (h), (p) and (s)) refutes a tiling formula $\psi^{}_\mathcal{W}$, then $\mathcal{W}$ tiles the plane. 

Note the weakness of these conditions. For instance, full associativity $(x\cdot y)\cdot z = x\cdot (y\cdot z)$ need not be assumed, only the one-way form (p) $(x\cdot y)\cdot z\subseteq x\cdot (y\cdot z)$. Nor do we require the frame conditions corresponding to Contraction or Pseudo-modus ponens, but only the condition (h).
\\\\
In contrast to the systems set out above, the semilattice logic $\mathsf{S}$ is defined semantically, and so is the logic of set frames $\mathsf{SetFr}$, which also falls within the scope of our undecidability result.

To define the latter, by a \textit{Routley--Meyer set frame}, we mean an $\mathsf{R}^+$-frame $\mathfrak{F}=(K, \cdot, N)$ that additionally satisfies
\begin{align*}
    x\cdot x\subseteq N\cdot x.
\end{align*}
The logic of set frames $\mathsf{SetFr}$ is the set of formulas valid in the class of Routley--Meyer set frames; that is, $\mathsf{SetFr}=\mathrm{Log}(\mathcal{S}et\mathcal{F}r)$, where $\mathcal{S}et\mathcal{F}r$ denotes the class of all such frames. The name is motivated by its role in the setting of \textit{collection frames}, where it was introduced by Restall and Standefer~\cite{RestallStandefer2023} as the logic arising from precisifying an informal notion of `collection' with the mathematical notion of a set.

The former, the semilattice logic $\mathsf{S}$, admits a particularly neat frame semantics, introduced in~\cite{Urquhart72, Urquhart73}:

\begin{Definition}
    A \textit{semilattice frame} for $\mathcal{L}$ is a triple $\mathfrak{F}=(S,\sqcup, 0)$ where $(S, \sqcup)$ is a semilattice with an identity element (a least element) $0  \in S$, i.e., $\sqcup:S\times S\to S$ is a function satisfying:
\begin{alignat*}{2}
  &\bullet\quad \text{Commutativity:} \quad & x \sqcup y &= y \sqcup x, \\
  &\bullet\quad \text{Associativity:} \quad & (x \sqcup y) \sqcup z &= x \sqcup (y \sqcup z), \\
  &\bullet\quad \text{Idempotence:} \quad & x \sqcup x &= x, \\
  &\bullet\quad \text{Identity:} \quad & 0 \sqcup x &= x.
\end{alignat*}
    
    A \textit{semilattice model} is a pair $\mathfrak{M}=(\mathfrak{F}, V)$ where $\mathfrak{F}$ is a semilattice frame and $V:\mathsf{P}\to \mathcal{P}(S)$ a function.

    The semantic clauses coincide with the clauses of Definition~\ref{def:clauses}; in particular, the clause for implication simplifies to
    \begin{align*}
        &\mathfrak{M}, x\Vdash \varphi\to\psi && \text{iff} && \text{for all $y\in S$: if } \mathfrak{M}, y\Vdash \varphi\text{, then } \mathfrak{M}, x\sqcup y\Vdash \psi.
    \end{align*}
    $\mathsf{S}$ is the set of validities over the class of semilattice frames, where a formula $\varphi\in \mathcal{L}$ is \textit{valid on a semilattice frame} $\mathfrak{F}=(S,\sqcup, 0)$ if for all models $\mathfrak{M}=(\mathfrak{F}, V)$ over $\mathfrak{F}$, $\varphi$ is satisfied at $0$, i.e., $\mathfrak{M}, 0\Vdash \varphi$.
\end{Definition}

\begin{Remark}
 Observe that a semilattice model $(S,\sqcup, 0, V)$ is a Routley--Meyer model $(S, \cdot, N, V)$ where 
        $$x\cdot y\ni z \qquad \text{iff}\qquad  x\sqcup y =z$$
    and 
        $$N\mathrel{:=}\{0\}.$$
    It is readily verified that a semilattice frame validates the correspondences of Proposition~\ref{pr:correspondence} along with $x\sqcup x=x= 0\sqcup x$, hence a semilattice frame is in particular a Routley--Meyer set frame and we have $\mathsf{SetFr}\subseteq \mathsf{S}$.\footnote{In fact, the inclusions $\mathsf{R}^+\subseteq \mathsf{SetFr}\subseteq \mathsf{S}$ are both strict. As shown by Standefer \cite{Standefer2022}, the formula $[(p\to q\lor r)\land (q\to r)]\to (p\to r)$ -- first identified by Dunn and Meyer to show that $\mathsf{R}^+\subsetneq \mathsf{S}$ -- witnesses the strictness of the latter inclusion, as it can be seen valid on all semilattice frames, but fails on a set frame. 
    
    The strictness of the former was posed as an open question by Restall and Standefer~\cite{RestallStandefer2023}, and settled by the author in recent unpublished work~\cite{Knudstorp:setframes}, identifying a separating formula $\phi\in \mathsf{S}\setminus \mathsf{SetFr}$. Letting $\circ$ bind stronger than $\land$, the formula is 
    \begin{align*}
        \phi&\mathrel{:=}\big[p\land (q\to r\lor s)\big]\to \big[q\to (\psi_r\lor\psi_s)\big],
    \end{align*}
    where
    \begin{align*}
        \psi_r&\mathrel{:=}r\land p\circ [(r\lor s)\land q\circ r],\\
        \psi_s&\mathrel{:=}s\land p\circ [(r\lor s)\land q\circ s].
    \end{align*}
    }
\end{Remark} 

We conclude this subsection with a semilattice frame of particular interest. For more on the semilattice semantics, see the survey~\cite{Standefer2022}.\footnote{Other recent work on the semilattice semantics includes that of Weiss~\cite{Weiss2019note,Weiss2021characteristic,Weiss21,Weiss2021reinterpretation}. \cite{Weiss2021characteristic} shows that $\mathsf{S}$ is complete for the semilattice frame $(\mathbb{N}, \mathrm{lcm}, 1)$ of natural numbers ordered by divisibility, and~\cite{Weiss2019note} uses the same frame to prove that $\mathsf{S}$ has the variable-sharing property. We will have further occasion to refer to~\cite{Weiss21,Weiss2021reinterpretation} below.}

\begin{Example}\label{ex:finitepower}
    Let $\mathcal{P}_{\mathrm{fin}}(\mathbb{N})$ be the set of all finite subsets of the natural numbers. Then
    \(
       (\mathcal{P}_{\mathrm{fin}}(\mathbb{N}), \cup, \varnothing)
    \)
    is a semilattice frame, where $\cup$ is the binary union operation, and the clause for implication (relative to a valuation) is
    \begin{align*}
        &x\Vdash \varphi\to\psi && \text{iff} && \text{for all $y$: if }y\Vdash \varphi\text{, then } x\cup y\Vdash \psi.
    \end{align*}
    This frame plays a key role in the undecidability proof: in Lemma~\ref{lm:tilingpowerset} we show that for every tile set $\mathcal{W}$, if $\mathcal{W}$ tiles the plane, then $(\mathcal{P}_{\mathrm{fin}}(\mathbb{N}), \cup, \varnothing)$ 
    refutes the corresponding tiling formula $\psi^{}_\mathcal{W}$ (and hence $\mathsf{S}\nvdash \psi^{}_\mathcal{W}$).
\end{Example}

\subsection{The tiling problem}\label{subsec:tilingproblem}
In the remainder of this section, we review the tiling problem used in the undecidability proof. 
\begin{Definition}[Wang tiling]
    A \textit{(Wang) tile} is a quadruple $(t_W, t_E, t_N,t_S)\in \mathbb{N}^4$, thought of as a square with edge colours $t_W, t_E, t_N,t_S$.

    Given a finite collection of tiles $\mathcal{W}\subseteq \mathbb{N}^4$ and a subset $D\subseteq \mathbb{Z}^2$, a \textit{$\mathcal{W}$-tiling of $D$} is a map 
    \[
        \tau:D\to \mathcal{W}, \qquad D\ni s\mapsto (\tau_W(s), \tau_E(s),\tau_N(s),\tau_S(s))\in \mathcal{W}
    \]
    such that adjacent tiles have matching colours on common sides, i.e., for all $(m,n), (m+1, n), (m, n+1)\in D$:
    \begin{align}
        \tau_E(m,n)=\tau_W(m+1,n)\qquad \text{and}\qquad \tau_N(m,n)=\tau_S(m,n+1).\label{eq:tilingCond}
    \end{align}
    We say that $\mathcal{W}$ \textit{tiles $D$} if there exists a $\mathcal{W}$-tiling of $D$.
\end{Definition}
The tiling problem for $\mathbb{Z}^2$ was introduced by Wang \cite{Wang1961,Wang1963}, and it is a classical result of Berger \cite{Berger} that it is undecidable.
\begin{Theorem}[The tiling problem]
    There is no Turing machine that, given the specifications of an arbitrary finite set of Wang tiles $\mathcal{W}$, decides whether $\mathcal{W}$ tiles $\mathbb{Z}^2$.
\end{Theorem}
We prove undecidability of our logical systems by reduction from the $\mathbb{Z}^2$ tiling problem. The reduction makes essential use of the equivalence between tiling $\mathbb{Z}^2$ and tiling finite isosceles right triangular regions of arbitrarily large size, which we call \textit{octants}. To state the equivalence, for $k\in \mathbb{N}$ we denote the \textit{$k$-quadrant} 
\[
    \mathbb{Q}(k)\mathrel{:=}[0,k]\times [0,k]\mathrel{=}\{(m,n)\mid 0\leq m,n\leq k\}
\]
and the \textit{$k$-octant} 
\[
    \mathbb{O}(k)\mathrel{:=}\{(m,n)\mid 0\leq m\leq  n\leq k\}.
\]
We then have the following.
\begin{Lemma}\label{lm:octanttoplane}
    For any finite collection of tiles $\mathcal{W}$, the following are equivalent:
    \begin{enumerate}
        \item $\mathcal{W}$ tiles $\mathbb{Z}^2$.
        \item $\mathcal{W}$ tiles $\mathbb{N}^2$.
        \item For all $k\in \mathbb{N}$, $\mathcal{W}$ tiles $\mathbb{Q}(k)$.
        \item For all $k\in \mathbb{N}$, $\mathcal{W}$ tiles $\mathbb{O}(k)$.
    \end{enumerate}
\end{Lemma}
\begin{proof}
    Since $\mathbb{O}(k)\subset \mathbb{Q}(k)\subset \mathbb{N}^2 \subset \mathbb{Z}^2$, we have that (1) implies (2) implies (3) implies (4). 
    
    Further, tiling $\mathbb{O}(2k)$ entails tiling 
    \[
        [0,k]\times [k,2k]=\{(m,n)\mid 0\leq m\leq k\leq n\leq 2k\},
    \] 
    which by translation with $(0,-k)$ entails tiling 
    \[
       [0,k]\times [0,k]=\mathbb{Q}(k),  
    \] hence (4) implies (3). 
    
    From (3), it follows that, for all $k\in \mathbb{N}$, $\mathcal{W}$ tiles the $2k$-quadrant centered at the origin 
    \[
        [-k, k]\times [-k, k]\mathrel{=}\{(m,n)\mid -k\leq m,n\leq k\},
    \]
    which by a Kőnig's Lemma argument implies (1): the nodes of the tree are tilings 
    \[
    \tau:[-n,n]\times [-n, n]\to \mathcal{W},
    \]
    and there is an edge from a node $\tau:[-n,n]^2\to \mathcal{W}$ to a node $\tau':[-n-1,n+1]^2\to \mathcal{W}$ just in case $\tau'$ agrees with $\tau$ on $[-n,n]^2$, i.e., $\tau'{\upharpoonright}_{[-n,n]^2}=\tau$.
\end{proof}
To encode the tiling problem, given a finite tile set $\mathcal{W}$, we effectively associate a formula $\psi^{}_{\mathcal{W}}\in\mathcal{L}$ and prove that 
\begin{enumerate}[label=(\roman*)]
	\item if $\mathsf{TWJ}^+\nvdash \psi^{}_{\mathcal{W}}$, then $\mathcal{W}$ tiles $\mathbb{O}(k)$ for all $k\in\mathbb{N}$, and
	\item if $\mathcal{W}$ tiles $\mathbb{Z}^2$, then $(\mathcal{P}_{\mathrm{fin}}(\mathbb{N}), \cup, \varnothing)\nvDash \psi^{}_{\mathcal{W}}$.
\end{enumerate}
Because of Lemma~\ref{lm:octanttoplane}, this establishes the undecidability of every set of formulas $\mathsf{L}\subseteq \mathcal{L}$ in the interval 
\[
    \mathsf{L}\in [\mathsf{TWJ}^+, \mathrm{Log}(\mathcal{P}_{\mathrm{fin}}(\mathbb{N}), \cup, \varnothing)],
\]
where $\mathrm{Log}(\mathcal{P}_{\mathrm{fin}}(\mathbb{N}), \cup, \varnothing)=\{\varphi \in\mathcal{L}\mid (\mathcal{P}_{\mathrm{fin}}(\mathbb{N}), \cup, \varnothing)\vDash \varphi\}$, since for such $\mathsf{L}$, (i) and (ii) imply
\[
    \psi^{}_{\mathcal{W}}\notin \mathsf{L}\qquad \text{if and only if}\qquad \text{$\mathcal{W}$ tiles $\mathbb{Z}^2$.}
\]
In particular, every logic whose $\mathcal{L}$-fragment contains $\mathsf{TWJ}^+$ and is sound for the frame $(\mathcal{P}_{\mathrm{fin}}(\mathbb{N}), \cup, \varnothing)$ is undecidable; this, of course, includes each of $\mathsf{TWJ}^+$, $\mathsf{C}^+$, $\mathsf{T}^+$, $\mathsf{E}^+$, $\mathsf{R}^+$, $\mathsf{S}$, $\mathsf{SetFr}$ (as well as logics with additional vocabulary, like $\mathsf{TWJ}$, $\mathsf{C}$, $\mathsf{T}$, $\mathsf{E}$, $\mathsf{R}$ where negation `$\neg$' is present).

Of the directions (i) and (ii), the former is more delicate but also illuminates why the combination of Hypothetical syllogism, Prefixing and Suffixing (which axiomatize $\mathsf{TWJ}^+$ relative to the decidable $\mathsf{B}^+$) leads to undecidability through the interaction of their respective frame correspondents
\begin{alignat*}{2}
        &\text{(h)}  &\qquad x\cdot y&\subseteq x\cdot (x\cdot y), \\
        &\text{(p)}  &\qquad (x\cdot y)\cdot z&\subseteq x\cdot (y\cdot z),\\
        &\text{(s)}  &\qquad (x\cdot y)\cdot z&\subseteq y\cdot (x\cdot z).    
    \end{alignat*}
Specifically, for $\psi^{}_{\mathcal{W}}$-refuting frames, (h), (p) and (s) allow us to define the octants $\mathbb{O}(k)$ inside these frames, which will be key in the proof of (i). We next outline how this is achieved.

\section{Octants $\mathbb{O}(k)$ in (h)-(p)-(s) models}\label{sec:TWJandtiling} In this section, we explain how the frame conditions (h), (p) and (s) -- correspondents of Hypothetical syllogism, Prefixing and Suffixing -- relate to $\mathbb{O}(k)$ tilings. This serves three purposes: to clarify how these axioms make for undecidability, to provide intuition for our later proof that $\mathsf{TWJ}^+\nvdash \psi^{}_{\mathcal{W}}$ implies that $\mathcal{W}$ tiles $\mathbb{O}(k)$ for all $k\in\mathbb{N}$ (Lemma~\ref{lm:tiling}), and to supply an auxiliary lemma used in that proof.

We begin with the auxiliary lemma. 
\begin{Lemma}\label{lm:infChain}
    Let $(K, \cdot, N)$ be a frame satisfying
    \begin{alignat*}{2}
        &\textnormal{(h)}  &\qquad x\cdot y&\subseteq x\cdot (x\cdot y), \\
        &\textnormal{(p)}  &\qquad (x\cdot y)\cdot z&\subseteq x\cdot (y\cdot z).    
    \end{alignat*}
    If there are elements $g, a_1, b_1, \hdots, a_n, b_n$ in $K$ such that
    \begin{align*}
            g\cdot b_1\ni a_1, \quad a_1\cdot b_2\ni a_2,\quad 
            a_2\cdot b_3\ni a_3,\quad \hdots,\quad a_{n-1}\cdot b_n\ni a_n,
        \end{align*}
    then there are elements $a_1', \hdots, a_n'$ such that 
    \begin{align*}
            g\cdot b_1\ni a_1', \quad a_1'\cdot b_2\ni a_2',\quad 
            a_2'\cdot b_3\ni a_3',\quad \hdots,\quad a_{n-1}'\cdot b_n\ni a_n',
        \end{align*}
    and for all $1\leq i\leq n$:
    \[
        g\cdot a_i'\ni a_i.
    \]
\end{Lemma}
\begin{proof}
    We show the existence of elements $a_1', \hdots, a_n'$ by induction on $1\leq i\leq n$.
    
    For the induction start, from $g\cdot b_1\ni a_1$, it follows by (h) that $g\cdot (g\cdot b_1)\ni a_1$. That is, there is a point $a_1'$ s.t. $g\cdot a_1'\ni a_1$ and $g\cdot b_1\ni a_1'$. 
    
    For the inductive step, assume the induction hypothesis for some $1\leq i<n$. We show that it holds for $i+1$. By the induction hypothesis $(g\cdot a_i')\ni a_i$ and by assumption of the lemma $a_i\cdot b_{i+1}\ni a_{i+1}$. Thus, $(g\cdot a_i')\cdot b_{i+1}\ni a_{i+1}$, so by (p) $g\cdot (a_i'\cdot b_{i+1})\ni a_{i+1}$. That is, there is a point $a_{i+1}'$ such that $g\cdot a_{i+1}'\ni a_{i+1}$ and $a_i'\cdot b_{i+1}\ni a_{i+1}'$, completing the induction.
\end{proof}
\begin{Remark}\label{rm:finitevsinfiteo}
    As an interlude, observe that the points in the consequent of the preceding lemma need not be those of the antecedent, i.e. we can have $a_i'\neq a_i$. Although it may not yet be apparent, this is why we consider tilings of arbitrarily large \textit{finite} octants $\mathbb{O}(k)$ rather than the infinite octant $\{(m,n)\mid 0\leq m\leq n\}$.
\end{Remark}
We do not elaborate on the lemma's use here; instead, we note only that it aids in constructing the \textit{lower staircase} of octants $\mathbb{O}(k)$ within $\psi^{}_{\mathcal{W}}$-refuting (h)-(p)-(s) models, illustrated in Figure~\ref{fig:lowerstaircase} for the case $\mathbb{O}(3)$. In the figure, we depict $a\cdot b\ni c$ as the arrow \tikz[baseline={(a.base)}]{
  \node (a) at (0,0) {$a$};
  \node (c) at (1.5,0) {$c$};
  \draw[->] (a) -- node[midway, below, >=stealth] {$\mathrel{\cdot} b$} (c);
}. We also use vertical arrows, which, strictly speaking, represent the same relation as horizontal arrows, but help visualize the construction as forming a staircase. 

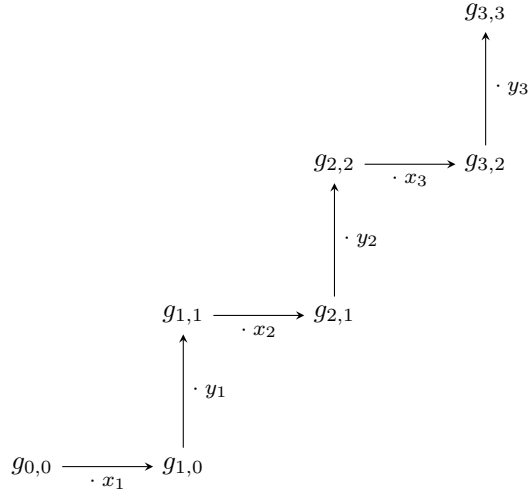
\begin{figure}[H]
\begin{center}
\begin{tikzpicture}[>=stealth, node distance=2cm]
  
  \node (a2) at (6,0) {$g_{0,0}$};
  \node (b2) at (8,0) {$g_{1,0}$};
  \node (c2) at (8,2) {$g_{1,1}$};
  \node (b3) at (10,2) {$g_{2,1}$};     
  \node (c3) at (10,4) {$g_{2,2}$};
  \node (g32) at (12,4) {$g_{3,2}$};     
  \node (g33) at (12,6) {$g_{3,3}$};
  

  \draw[->] (a2) -- node[midway,below]{\footnotesize $\mathrel{\cdot} x_1$} (b2);
  \draw[->] (b2) -- node[midway,right]{\footnotesize $\mathrel{\cdot} y_1$} (c2);
  \draw[->] (c2) -- node[midway,below]{\footnotesize $\mathrel{\cdot} x_2$} (b3);
  \draw[->] (b3) -- node[midway,right]{\footnotesize $\mathrel{\cdot} y_2$} (c3);
  \draw[->] (c3) -- node[midway,below]{\footnotesize $\mathrel{\cdot} x_3$} (g32);
  \draw[->] (g32) -- node[midway,right]{\footnotesize $\mathrel{\cdot} y_3$} (g33);
\end{tikzpicture}
\caption{The lower staircase for $\mathbb{O}(3)$; its construction relies on (h) and (p).}
  \label{fig:lowerstaircase}
\end{center}
\end{figure}
Applying $\text{(s) }   (x\cdot y)\cdot z\subseteq y\cdot (x\cdot z)$, we obtain the \textit{upper staircase} from the lower staircase, as illustrated in Figure~\ref{fig:lowertoupperstaircase} for $\mathbb{O}(3)$. The key observation is that $a\cdot b\ni c$ may be viewed not only as $a$ followed by right-composition with $b$, as suggested by the depiction \tikz[baseline={(a.base)}]{
  \node (a) at (0,0) {$a$};
  \node (c) at (1.5,0) {$c$};
  \draw[->] (a) -- node[midway, below, >=stealth] {$\mathrel{\cdot} b$} (c);
}, but also as $b$ followed by left-composition with $a$, which we instead depict as \tikz[baseline={(a.base)}]{
  \node (b) at (0,0) {$b$};
  \node (c) at (1.5,0) {$c$};
  \draw[->] (b) -- node[midway, below, >=stealth] {$a \mathrel{\cdot}$} (c);
}. Concretely, by applying $(s)$ we obtain upper-staircase points such as $g_{0,1}$ as follows:
\begin{align*}
    g_{0,0}\cdot x_{1}\ni g_{1,0}\text{ and }g_{1,0}\cdot y_{1}\ni g_{1,1}\quad &\Rightarrow\quad (g_{0,0}\cdot x_{1})\cdot y_{1}\ni g_{1,1} \quad \overset{\textnormal{(s)}}{\Rightarrow}\quad x_{1}\cdot(g_{0,0}\cdot y_{1})\ni g_{1,1}\\
    &\Rightarrow\quad \text{there is a witness, denoted $g_{0,1}$, with }\\
    &\;\;\;\;\;\, \quad x_{1}\cdot g_{0,1}\ni g_{1,1}\text{ and }g_{0,0}\cdot y_{1}\ni g_{0,1}.
\end{align*}

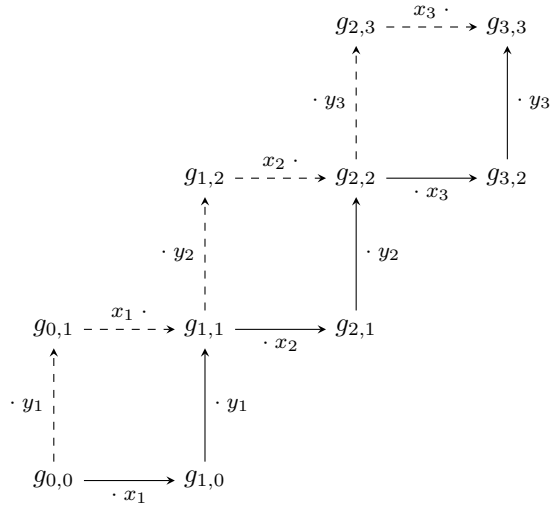
\begin{figure}[H]
\begin{center}
\begin{tikzpicture}[>=stealth, node distance=2cm]
  \node (a2) at (6,0) {$g_{0,0}$};
  \node (b2) at (8,0) {$g_{1,0}$};
  \node (c2) at (8,2) {$g_{1,1}$};
  \node (b3) at (10,2) {$g_{2,1}$};     
  \node (c3) at (10,4) {$g_{2,2}$};
  \node (g32) at (12,4) {$g_{3,2}$};     
  \node (g33) at (12,6) {$g_{3,3}$};
  
  \node (d2) at (6,2) {$g_{0,1}$};
  \node (d3) at (8,4) {$g_{1,2}$};
  \node (g23) at (10,6) {$g_{2,3}$}; 
  

  \draw[->] (a2) -- node[midway,below]{\footnotesize $\mathrel{\cdot} x_1$} (b2);
  \draw[->] (b2) -- node[midway,right]{\footnotesize $\mathrel{\cdot} y_1$} (c2);
  \draw[->] (c2) -- node[midway,below]{\footnotesize $\mathrel{\cdot} x_2$} (b3);
  \draw[->] (b3) -- node[midway,right]{\footnotesize $\mathrel{\cdot} y_2$} (c3);
    \draw[->] (c3) -- node[midway,below]{\footnotesize $\mathrel{\cdot} x_3$} (g32);
  \draw[->] (g32) -- node[midway,right]{\footnotesize $\mathrel{\cdot} y_3$} (g33);

  \draw[->, dashed] (a2) -- node[midway,left]{ \footnotesize $\mathrel{\cdot} y_1$ } (d2);
  \draw[->, dashed] (d2) -- node[midway,above]{\footnotesize $x_1\mathrel{\cdot}$} (c2);
    \draw[->, dashed] (d3) -- node[midway,above]{\footnotesize $x_2\mathrel{\cdot}$} (c3);
    \draw[->, dashed] (c2) -- node[midway,left]{ \footnotesize $\mathrel{\cdot} y_2$ } (d3);
    \draw[->, dashed] (g23) -- node[midway,above]{\footnotesize $x_3\mathrel{\cdot}$} (g33);
    \draw[->, dashed] (c3) -- node[midway,left]{ \footnotesize $\mathrel{\cdot} y_3$ } (g23);
    \end{tikzpicture}
\caption{The upper staircase for $\mathbb{O}(3)$, constructed via (s).\\ Note the distinction between an edge being labelled with a right-composition, as in $\cdot \,x_1$, and with a left-composition, as in $x_1\cdot$.
}
  \label{fig:lowertoupperstaircase}
\end{center}
\end{figure}
Finally, the octant $\mathbb{O}(3)$ is constructed from the upper staircase by means of $\text{(p) }   (x\cdot y)\cdot z\subseteq x\cdot (y\cdot z)$, as in Figure~\ref{fig:upperstaircasetooctant}.

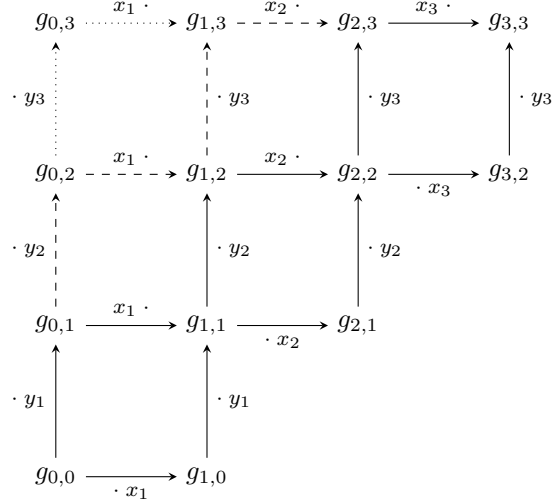
\begin{figure}[H]
\begin{center}
\begin{tikzpicture}[>=stealth, node distance=2cm]
  \node (a2) at (6,0) {$g_{0,0}$};
  \node (b2) at (8,0) {$g_{1,0}$};
  \node (c2) at (8,2) {$g_{1,1}$};
  \node (b3) at (10,2) {$g_{2,1}$};     
  \node (c3) at (10,4) {$g_{2,2}$};
  \node (g32) at (12,4) {$g_{3,2}$};     
  \node (g33) at (12,6) {$g_{3,3}$};
  
  \node (d2) at (6,2) {$g_{0,1}$};
  \node (d3) at (8,4) {$g_{1,2}$};
    \node (g23) at (10,6) {$g_{2,3}$};

  \node (g02) at (6,4) {$g_{0,2}$};
  \node (g03) at (6,6) {$g_{0,3}$};
  \node (g13) at (8,6) {$g_{1,3}$};
  

  \draw[->] (a2) -- node[midway,below]{\footnotesize $\mathrel{\cdot} x_1$} (b2);
  \draw[->] (b2) -- node[midway,right]{\footnotesize $\mathrel{\cdot} y_1$} (c2);
  \draw[->] (c2) -- node[midway,below]{\footnotesize $\mathrel{\cdot} x_2$} (b3);
  \draw[->] (b3) -- node[midway,right]{\footnotesize $\mathrel{\cdot} y_2$} (c3);
    \draw[->] (c3) -- node[midway,below]{\footnotesize $\mathrel{\cdot} x_3$} (g32);
  \draw[->] (g32) -- node[midway,right]{\footnotesize $\mathrel{\cdot} y_3$} (g33);

    \draw[->] (a2) -- node[midway,left]{ \footnotesize $\mathrel{\cdot} y_1$ } (d2);
  \draw[->] (d2) -- node[midway,above]{\footnotesize $x_1\mathrel{\cdot}$} (c2);
    \draw[->] (d3) -- node[midway,above]{\footnotesize $x_2\mathrel{\cdot}$} (c3);
    \draw[->] (c2) -- node[midway,right]{ \footnotesize $\mathrel{\cdot} y_2$ } (d3);
    \draw[->] (g23) -- node[midway,above]{\footnotesize $x_3\mathrel{\cdot}$} (g33);
    \draw[->] (c3) -- node[midway,right]{ \footnotesize $\mathrel{\cdot} y_3$ } (g23);
  
  \draw[->, dashed] (g02) -- node[midway,above]{\footnotesize $x_1\mathrel{\cdot}$} (d3);
  \draw[->, dashed] (d2) -- node[midway,left]{\footnotesize $\mathrel{\cdot} y_2$} (g02);
  \draw[->, dashed] (g13) -- node[midway,above]{\footnotesize $x_2\mathrel{\cdot} $} (g23);
  \draw[->, dashed] (d3) -- node[midway,right]{\footnotesize $\mathrel{\cdot} y_3$} (g13);
  
  \draw[->, dotted] (g03) -- node[midway,above]{\footnotesize $x_1\mathrel{\cdot}$} (g13);
  \draw[->, dotted] (g02) -- node[midway,left]{\footnotesize $\mathrel{\cdot} y_3$} (g03);
    \end{tikzpicture}
\caption{The octant $\mathbb{O}(3)$, constructed recursively from the upper staircase via (p): first the dashed $\protect\tikz[baseline=-0.5ex]{\protect\draw[->, dashed, >=stealth] (0.5,0) -- (1,0);}$, then the dotted $\protect\tikz[baseline=-0.5ex]{\protect\draw[->, dotted, >=stealth] (0.5,0) -- (1,0);}$.}
  \label{fig:upperstaircasetooctant}
\end{center}
\end{figure}

\begin{Remark}[Sources of (un)decidability]
    This highlights the role of Hypothetical syllogism, Prefixing and Suffixing in driving undecidability. In comparison, recall that omitting Hypothetical syllogism from $\mathsf{TWJ}^+$ gives the \textit{decidable} logic $\mathsf{TW}^+$. 
    This reflects a more general phenomenon: its (Hypothetical syllogism-free) extension $\mathsf{RW}^+=\mathsf{TW}^+\oplus \text{Assertion}$ is likewise decidable~\cite{Giambrone1985}, and as are related Hypothetical syllogism-free (usually presented as Contraction-free) variants of $\mathsf{S}$~\cite{GiambroneKron1987}. Similarly, when Hypothetical syllogism is retained but the affixing axioms are omitted, the resulting logic $\mathsf{B}^+\oplus \text{Hypothetical syllogism}$ is decidable~\cite{Brady2003}.

    We should also mention that while the combination of these three axioms drives undecidability modulo the decidable basic relevant logic $\mathsf{B}^+$, further contributions, of course, come from axioms already present in $\mathsf{B}^+$. For example, omitting distributivity from $\mathsf{R}^+$ yields Lattice-$\mathsf{R}^+$, which is decidable~\cite{Meyer1966} (although it has no primitive recursive decision procedure~\cite{Urquhart1999}
    ).
    
    The focus so far has been on the lower bound of the undecidable interval $[\mathsf{TWJ^+}, \mathrm{Log}(\mathcal{P}_{\mathrm{fin}}(\mathbb{N}), \cup, \varnothing)]$. But instead of leaving out theorems of $\mathsf{TWJ}^+$, decidability may also be brought about by extending with formulas not valid in $(\mathcal{P}_{\mathrm{fin}}(\mathbb{N}), \cup, \varnothing)$. 
    For instance, adding $\varphi\to(\psi \to \varphi)$ to $\mathsf{R}^+$ results in the $\{\land,\lor,\to\}$-fragment of intuitionistic propositional logic, while adding Peirce's law $((\varphi\to\psi)\to \varphi)\to \varphi$ results in the corresponding fragment of classical propositional logic. If, however, we add to $\mathsf{R}$, in the language with negation, another paradox of material implication, namely explosion $\varphi\land\neg \varphi\to \psi$, we obtain the logic $\mathsf{KR}$, which remains undecidable~\cite{urquhart84:jsl}. 
    
    Lastly, from a semantical viewpoint, it is of interest to note how decidability can arise through restricting the set of admissible valuations: if semilattice valuations are required to be hereditary (i.e., if $x \in V(p)$, then $x \sqcup y \in V(p)$), one obtains a complete semilattice semantics for the $\{\land,\lor,\to\}$-fragment of intuitionistic propositional logic~\cite{Weiss2021reinterpretation,WeissStandefer2026}. This parallels a point made in~\cite{Knudstorp2025a}.
\end{Remark}

\section{The tiling formulas}\label{sec:tilingformulas} We now define the tiling formulas $\psi^{}_{\mathcal{W}}$ associated with a finite set of tiles $\mathcal{W}$. High-level intuition will be given afterward.

\begin{Definition}[Tiling formulas]   
    Let a finite tile set $\mathcal{W}$ be given. For each tile $t\in \mathcal{W}$, we introduce a propositional letter $\mathtt{t}$, and define formulas for the permissible rightwards and upwards neighbours, respectively,
    \[
        \mathtt{R}(t)\mathrel{:=}\bigvee_{t'\in\mathcal{W},\, t^{}_E\,=\,t'_W}\mathtt{t'},\qquad \mathtt{U}(t)\mathrel{:=}\bigvee_{t'\in\mathcal{W},\, t^{}_N\,=\,{t'_S}}\mathtt{t'}.
    \]
    We further include the propositional letters $\mathtt{x}$, $\mathtt{y}$, $\mathtt{p}_\top$, and for each of the four combinations $i,j\in \{0,1\}$, the letter $\mathtt{m_{i,j}}$. We then abbreviate
    \begin{align*}
        \mathtt{G_{i,j}}&\mathrel{:=} \mathtt{m_{i,j}}\land \bigvee_{t\in\mathcal{W}}\left[\mathtt{t}\land [\mathtt{y}\to(\mathtt{m_{i,j}}\lor (\mathtt{m_{i,1-j}} \land \mathtt{U}(t)))]\right]\\
        \mathtt{x'}&\mathrel{:=}\mathtt{x}\land \mathtt{p_\top}\land (\mathtt{p_\top}\to \mathtt{p_\top})\land \bigwedge_{\substack{t\in \mathcal{W}, \\ i,j\in \{0,1\}}}\left[(\mathtt{m_{i,j}}\land \mathtt{t})\to (\mathtt{m_{i,j}}\lor (\mathtt{m_{1-i,j}}\land \mathtt{R}(t)))\right]\\
        &\;\;\;\;\;\,\land \bigwedge_{(i,j)\neq (i',j')}\left[(\mathtt{m_{i,j}}\land \mathtt{m_{i',j'}})\to (\mathtt{m_{0,0}}\land \mathtt{m_{0,1}})\right]\land \bigwedge_{\substack{t,t'\in \mathcal{W}, \\ t\neq t'}}\left[(\mathtt{t}\land \mathtt{t}')\to (\mathtt{m_{0,0}}\land \mathtt{m_{0,1}})\right]\\ \mathtt{y'}&\mathrel{:=}\mathtt{y}\land \mathtt{p_\top}\land (\mathtt{p_\top}\to \mathtt{p_\top})\\
        \mathtt{m_{i,j}^c}&\mathrel{:=} \mathtt{m_{i,1-j}}\lor \mathtt{m_{1-i,j}}\lor \mathtt{m_{1-i,1-j}}\qquad\qquad\qquad\qquad\qquad\, \text{(a notational complement)}\\
        \alpha&\mathrel{:=}\bigwedge_{i,j\in\{0,1\}}\left(\mathtt{m_{i,j}^c}\to \mathtt{m_{i,j}^c}\right)\\
        \beta_1&\mathrel{:=} \bigwedge_{i\in\{0,1\}}\left[(\mathtt{y'}\to \mathtt{m_{i,i}^c})\to \mathtt{m_{i,1-i}^c}\right]\\ 
        \beta_2&\mathrel{:=} \bigwedge_{i\in\{0,1\}}\left[(\mathtt{x'}\to \mathtt{m_{1-i,i}^c})\to \mathtt{m_{i,i}^c}\right]\\
        \gamma&\mathrel{:=}\mathtt{p_\top}\to \bigvee_{i,j\in \{0,1\}}\mathtt{G_{i,j}}.
    \end{align*}    
Finally, we let
\[
    \psi^{}_\mathcal{W} \mathrel{:=} \left(\alpha \land\beta_1\land \beta_2\land  \gamma\land \mathtt{G_{0,0}}\right)\to (\mathtt{x'}\to \mathtt{m_{1,0}^c}).\footnotemark
\]
\end{Definition}
With\footnotetext{The conjunct `$\mathtt{G_{0,0}}$' is actually not needed. It is included as it slightly simplifies a part of the undecidability proof.} the tiling formulas defined, we proceed to clarify their underlying intuition, meant only as a guide -- detailed arguments are left to the proofs. 

The propositional letters $\mathtt{t}$ for $t\in\mathcal{W}$, as well as $\mathtt{R}(t)$ and $\mathtt{U}(t)$, are self-explanatory. For the remaining propositional letters, we offer the following intuition:
    \begin{itemize}
        \item $\mathtt{m_{i,j}}$ records the parity of the $x$- and $y$-coordinate of elements that represent grid points $(m,n)$: $\mathtt{i}$ for the parity of $m$, and $\mathtt{j}$ for the parity of $n$.
        \item $\mathtt{x}$ and $\mathtt{y}$ mark elements that can effect unit increments in the $x$- and $y$-coordinate, respectively, analogous to the generators $(1,0)$ and $(0,1)$. 
        \item $\mathtt{p}_\top$ acts as `verum' and is solely used to make elements satisfy the grid formulas $\mathtt{G_{i,j}}$ (see $\gamma$).
    \end{itemize}
Regarding the design of the tiling formulas
\[
    \psi^{}_\mathcal{W} \mathrel{=} \left(\alpha \land\beta_1\land \beta_2\land  \gamma\land \mathtt{G_{0,0}}\right)\to (\mathtt{x'}\to \mathtt{m_{1,0}^c}),
\]
note that, if such formula is refuted in a model, there must be an element, to be denoted $g_{0,0}$, that satisfies the antecedent while refuting the consequent (cf. Remark~\ref{rm:verification}). The refutation of the consequent initiates the construction of the lower staircases for the octants $\mathbb{O}(k)$. 
The satisfaction of the antecedent, in turn, guarantees (a) the continuation of the construction of the lower staircases -- which, as outlined earlier, generate the octants $\mathbb{O}(k)$ -- and (b) the assignment of tiles to octant points in accordance with the tiling conditions~\eqref{eq:tilingCond}. 

(a) is ensured by the formulas $\alpha$ and $\beta_i$, while (b) is ensured partly by $\gamma$ and the grid formulas $\mathtt{G_{i,j}}$, and partly by $\mathtt{x'}$. The $\beta_i$'s function much like the consequent of $\psi^{}_{\mathcal{W}}$ in continuing the lower-staircase construction, whereas $\alpha$ is included for technical reasons related to Lemma~\ref{lm:infChain} and the subsequent Remark~\ref{rm:finitevsinfiteo}. $\gamma$ ensures that elements satisfy a grid formula 
\[
    \mathtt{G_{i,j}}\mathrel{=} \mathtt{m_{i,j}}\land \bigvee_{t\in\mathcal{W}}\left[\mathtt{t}\land [\mathtt{y}\to(\mathtt{m_{i,j}}\lor (\mathtt{m_{i,1-j}} \land \mathtt{U}(t)))]\right].
\]
The grid formulas record the parity of the $x$- and $y$-coordinate [via $\mathtt{m_{i,j}}$], assign a tile [via $\mathtt{t}$ for some $t\in \mathcal{W}$] and require that $y$-increments [expressed by the antecedent $\mathtt{y}\to\cdots$]  result in a point with appropriate parity and matching tile [specified by the second disjunct in the consequent $\cdots\to(\mathtt{m_{i,j}}\lor (\mathtt{m_{i,1-j}} \land \mathtt{U}(t))$]. 

The first disjunct in the consequent strengthens the tiling formula $\psi^{}_\mathcal{W}$ so that whenever $\mathcal{W}$ tiles $\mathbb{Z}^2$, the formula can be refuted in $(\mathcal{P}_{\mathrm{fin}}(\mathbb{N}), \cup, \varnothing)$; it is not needed for the converse direction. This also illustrates a balance in the design of $\psi^{}_\mathcal{W}$: it must be weak enough to extract a tiling from its refutation, yet strong enough to be refutable whenever a tiling exists.

While grid formulas constrain $y$-increments, they do not impose an analogous constraint on $x$-increments. Although the tiling conditions are symmetric, they are encoded in an asymmetric setup (as $\to$ takes arguments on the right, with no corresponding connective taking arguments on the left);\footnote{See~\cite{Knudstorp2025a} and~\cite{GalatosJipsenKnudstorpRevantha:BIund} for settings that draw on similar tiling-based undecidability technology but admit a more straightforward symmetric encoding.} the horizontal constraint is therefore coded separately in $\mathtt{x'}$, in its fourth conjunct. Its second and third conjuncts, $\mathtt{p_\top}\land (\mathtt{p_\top}\to \mathtt{p_\top})$, propagate the satisfaction of $\mathtt{p_\top}$, while its fifth and sixth conjuncts ensure that octant points have a unique parity [i.e. satisfy at most one $\mathtt{m_{i,j}}$] and are assigned a unique tile [i.e. satisfy at most one $\mathtt{t}$].

The latter uniqueness conditions will follow inductively via `modus tollens': in general, we shall repeatedly use that if $a\Vdash \varphi\to\psi$, $a\cdot b \ni c$ and $c\nVdash \psi$, then, by \textit{modus tollens}, $b\nVdash \varphi$.

\section{The tiling lemmas}\label{sec:tilinglemmas} 
Since the construction of $\psi^{}_\mathcal{W}$ from $\mathcal{W}$ is computable, the undecidability proof comes down to the two tiling lemmas, stated earlier in \S\ref{subsec:tilingproblem} as (i) and (ii). We present each lemma together with its proof in separate subsections.

\subsection{The first tiling lemma}
We begin with the first of the two tiling lemmas.
\begin{Lemma}\label{lm:tiling}
    If $\mathsf{TWJ}^+\nvdash\psi^{}_\mathcal{W}$, then $\mathcal{W}$ tiles $\mathbb{O}(k)$ for all $k\in \mathbb{N}$.
\end{Lemma}
\begin{proof}
    Let $\mathcal{W}$ be an arbitrary finite tile set, and suppose $\mathsf{TWJ}^+\nvdash\psi^{}_\mathcal{W}$. Then by completeness, there is an (h), (p), (s) frame $\mathfrak{F}=(K, \cdot, N)$ that refutes $\psi^{}_\mathcal{W}$. Hence there is a model $\mathfrak{M}=(\mathfrak{F}, V)$ over $\mathfrak{F}$ and a point, which we denote $g_{0,0}\in K$, such that
        $$\mathfrak{M}, g_{0,0}\Vdash\alpha \land \beta_1\land \beta_2\land  \gamma\land \mathtt{G_{0,0}}, $$ 
    but
        $$\mathfrak{M}, g_{0,0}\nVdash \mathtt{x'}\to\mathtt{m_{1,0}^c}.$$
    From here, we show that for all $k\in \mathbb{N}$, $\mathcal{W}$ tiles the octant $\mathbb{O}(k)$. 
    Accordingly, let $k\in \mathbb{N}$ be arbitrary. The proof proceeds in four steps:
    \begin{enumerate}[label=\Roman*.]
        \item We construct a lower staircase within $\mathfrak{M}$ for $\mathbb{O}(k)$.
        \item We extend this to an upper staircase for $\mathbb{O}(k)$.
        \item We generate the octant $\mathbb{O}(k)$.
        \item We define a function $\tau:\mathbb{O}(k)\to\mathcal{W}$ and verify that it is a tiling.
    \end{enumerate}
    
    \textbf{Part I.} First, we find $x_1,\hdots, x_k; y_1, \hdots, y_k; g_{1,0}, g_{1,1},$ $g_{2,1}, g_{2,2}, \hdots, g_{k,k-1}, g_{k,k}$ with
    \begin{alignat}{3}
        g_{2i+1,2i} &\nVdash \mathtt{m_{1,0}^c},&\qquad g_{2i+1,2i+1} \nVdash \mathtt{m_{1,1}^c}, &\qquad x_i\Vdash \mathtt{x'},\label{eq:diagonal1} \\
        g_{2i,2i-1} &\nVdash \mathtt{m_{0,1}^c},&\qquad g_{2i,2i} \nVdash \mathtt{m_{0,0}^c}, &\qquad y_i\Vdash \mathtt{y'}, \notag
    \end{alignat}
    and
        \begin{align}
            g_{0,0}\cdot x_1&\ni g_{1,0},&\qquad g_{1,0}\cdot y_1&\ni g_{1,1},\label{eq:stair}\\
            g_{1,1}\cdot x_2&\ni g_{2,1},&\qquad g_{2,1}\cdot y_2&\ni g_{2,2},\notag\\
            &\vdots&&\vdots\notag\\
            g_{k-1,k-1}\cdot x_k&\ni g_{k,k-1},&\qquad g_{k,k-1}\cdot y_k&\ni g_{k,k}.\notag
        \end{align}        
    We find the points recursively and establish their properties by induction on the last sequence~\eqref{eq:stair}: 
    \[
        g_{0,0}\cdot x_1\ni g_{1,0},\quad g_{1,0}\cdot y_1\ni g_{1,1},\quad  g_{1,1}\cdot x_2\ni g_{2,1}, \quad\hdots
    \]
    The base of the induction follows from $g_{0,0}\nVdash \mathtt{x'}\to\mathtt{m_{1,0}^c}$. Indeed, by the semantic clause for $\to$, there are points, which we denote $x_1, g_{1,0}\in K$, such that $g_{0,0}\cdot x_1\ni g_{1,0}$ and $x_1\Vdash \mathtt{x'}$ but $g_{1,0}\nVdash \mathtt{m_{1,0}^c}$. (We do not claim, and do not need, $g_{0,0}\nVdash \mathtt{m_{0,0}^c}$.)

    The inductive step splits into two cases: (a) assuming the claim up to some $g_{n,n}\cdot x_{n+1}\ni g_{n+1,n}$, and (b) assuming the claim up to some $g_{n+1,n}\cdot y_{n+1}\ni g_{n+1,n+1}$. As the cases are analogous, we only cover the former. By Lemma~\ref{lm:infChain}, there are points $g'_{1,0}, g'_{1,1},$ $g'_{2,1}, g'_{2,2}, \hdots, g'_{n,n}, g'_{n+1,n}$ such that
    \begin{alignat*}{2}
        g_{0,0}\cdot x_1&\ni g'_{1,0},&\qquad g'_{1,0}\cdot y_1&\ni g'_{1,1},\\
        g'_{1,1}\cdot x_2&\ni g'_{2,1},&\qquad g'_{2,1}\cdot y_2&\ni g'_{2,2},\\
        &\vdots&\qquad &\vdots\\
        g'_{n-1,n-1}\cdot x_{n}&\ni g'_{n,n-1},&\qquad g'_{n,n-1}\cdot y_n&\ni g'_{n,n},\\
        g'_{n,n}\cdot x_{n+1}&\ni g'_{n+1,n},
    \end{alignat*}    
    and for each $g'_{i,j}$,
    \[
        g_{0,0}\cdot g'_{i,j}\ni g_{i,j}.
    \]
    From the latter, along with the fact that $g_{0,0}\Vdash \alpha$ and the inductive hypothesis, it follows by modus tollens that
    \begin{alignat*}{2}
        g'_{2i+1,2i} &\nVdash \mathtt{m_{1,0}^c},&\qquad g'_{2i+1,2i+1} \nVdash \mathtt{m_{1,1}^c},\\
        g'_{2i,2i-1} &\nVdash \mathtt{m_{0,1}^c},&\qquad g'_{2i,2i} \nVdash \mathtt{m_{0,0}^c}.
        \end{alignat*}    
    We now find the points $y_{n+1}$ and $g'_{n+1, n+1}$ required for the inductive step. If $n$ is even, we need $g'_{n+1,n+1}\nVdash  \mathtt{m_{1,1}^c}$, while if $n$ is odd, we need $g'_{n+1,n+1}\nVdash  \mathtt{m_{0,0}^c}$. As the two cases are analogous, we consider only the former and assume that $n$ is even. 
    
    By the inductive hypothesis, we know that $g_{n+1,n}\nVdash  \mathtt{m_{1,0}^c}$. Together with $g_{0,0}\cdot g'_{n+1,n}\ni g_{n+1,n}$ and $g_{0,0}\Vdash \beta_1$ (so in particular $g_{0,0} \Vdash(\mathtt{y'}\to \mathtt{m_{1,1}^c})\to \mathtt{m_{1,0}^c}$), this implies, by modus tollens, that
        \[
            g'_{n+1,n}\nVdash \mathtt{y'}\to \mathtt{m_{1,1}^c}.
        \]
    Hence there are witnessing points $y_{n+1}$ and $g'_{n+1, n+1}$ such that 
    \[
        y_{n+1}\Vdash \mathtt{y'}, \quad g'_{n+1, n+1}\nVdash \mathtt{m_{1,1}^c}, \quad \text{and}\quad  g'_{n+1,n}\cdot y_{n+1}\ni g'_{n+1, n+1},
    \]
    which completes the induction.

    \textbf{Part II.} The above constructs the lower staircase $g_{0,0},g_{1,0}, g_{1,1},\hdots, g_{k,k-1}, g_{k,k}$. For the upper staircase, observe that for all $0\leq i<k$,
    \[
        (g_{i,i}\cdot x_{i+1})\cdot y_{i+1}\ni g_{i+1, i+1},
    \]
    and hence, by (s), 
    \[
        x_{i+1}\cdot (g_{i,i}\cdot y_{i+1})\ni g_{i+1, i+1}.
    \]
    Thus there are points $g_{0,1}, g_{1,2}, g_{2,3}, \hdots, g_{k-1,k}$ such that
    \begin{alignat*}{2}
        g_{0,0}\cdot y_1&\ni g_{0,1},&\qquad x_1\cdot g_{0,1}&\ni g_{1,1},\\
        g_{1,1}\cdot y_2&\ni g_{1,2},&\qquad x_2\cdot g_{1,2}&\ni g_{2,2},\\
        &\vdots&\qquad &\vdots\\
        g_{k-1,k-1}\cdot y_k&\ni g_{k-1,k},&\qquad x_k\cdot g_{k-1,k}&\ni g_{k,k}.
    \end{alignat*}  
    This yields the upper staircase (see Figure~\ref{fig:upperstaircase}).

\begin{figure}[H]
\begin{center}
\begin{tikzpicture}[>=stealth, node distance=2cm]
  
  \node (a2) at (6,0) {$g_{0,0}$};
  \node (b2) at (8,0) {$g_{1,0}$};
  \node (c2) at (8,2) {$g_{1,1}$};
  \node (b3) at (10,2) {$g_{2,1}$};     
  \node (c3) at (10,4) {$g_{2,2}$};      \node at (11,5) {$\iddots$};
  \node (k-1) at (12,6) {$g_{k-1,k-1}$};
  \node (k-1h) at (14,6) {$g_{k,k-1}$};
  \node (k-1v) at (12,8) {$g_{k-1,k}$};
  \node (k) at (14,8) {$g_{k,k}$};
  
  \node (d2) at (6,2) {$g_{0,1}$};
  \node (d3) at (8,4) {$g_{1,2}$};      
  

  \draw[->] (a2) -- node[midway,below]{\footnotesize $\mathrel{\cdot} x_1$} (b2);
  \draw[->] (b2) -- node[midway,right]{\footnotesize $\mathrel{\cdot} y_1$} (c2);
  \draw[->] (c2) -- node[midway,below]{\footnotesize $\mathrel{\cdot} x_2$} (b3);
  \draw[->] (b3) -- node[midway,right]{\footnotesize $\mathrel{\cdot} y_2$} (c3);
  \draw[->] (k-1) -- node[midway,below]{\footnotesize $\mathrel{\cdot} x_k$} (k-1h);
  \draw[->] (k-1h) -- node[midway,right]{\footnotesize $\mathrel{\cdot} y_k$} (k);

  \draw[->, dashed] (a2) -- node[midway,left]{ \footnotesize $\mathrel{\cdot} y_1$ } (d2);
  \draw[->, dashed] (d2) -- node[midway,above]{\footnotesize $x_1\mathrel{\cdot}$} (c2);
    \draw[->, dashed] (d3) -- node[midway,above]{\footnotesize $x_2\mathrel{\cdot}$} (c3);
    \draw[->, dashed] (c2) -- node[midway,left]{ \footnotesize $\mathrel{\cdot} y_2$ } (d3);
    \draw[->, dashed] (k-1) -- node[midway,left]{\footnotesize $\mathrel{\cdot} y_k$} (k-1v);
  \draw[->, dashed] (k-1v) -- node[midway,above]{\footnotesize $x_k\mathrel{\cdot}$} (k);
\end{tikzpicture}
\caption{Extending the lower staircase to the upper staircase through $\text{(s)}  \qquad (a\cdot b)\cdot c\subseteq b\cdot (a\cdot c).$}
  \label{fig:upperstaircase}
\end{center}
\end{figure}
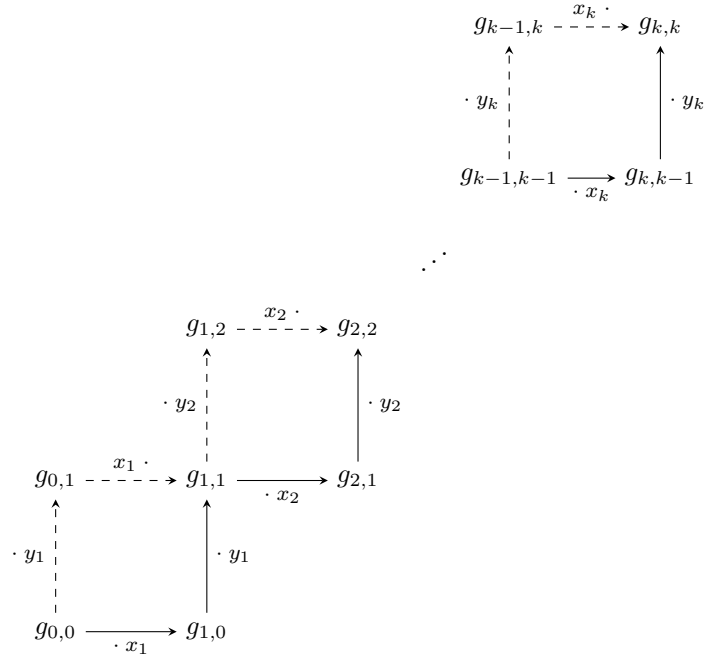

    \textbf{Part III.} To generate the full octant $\mathbb{O}(k)$, we need points $g_{i,j}$ for all $0\leq i\leq j\leq k$ with
    \begin{align}
        x_{i+1}\cdot g_{i,j}\ni g_{i+1, j}\qquad \text{and} \qquad g_{i,j}\cdot y_{j+1}\ni g_{i, j+1}.\label{eq:xOntheleft}
    \end{align}
    We obtain these points recursively. The base case is the upper staircase, where~\eqref{eq:xOntheleft} already holds. For the inductive step, suppose
    \[
        x_{i+1}\cdot g_{i,j}\ni g_{i+1, j}\qquad \text{and} \qquad g_{i+1,j}\cdot y_{j+1}\ni g_{i+1, j+1}.
    \]
    Then $(x_{i+1}\cdot g_{i,j})\cdot y_{j+1}\ni g_{i+1, j+1}$, so by (p), $x_{i+1}\cdot (g_{i,j}\cdot y_{j+1})\ni g_{i+1, j+1}$. Thus there is a point $g_{i,j+1}$ such that
    \[
        x_{i+1}\cdot g_{i,j+1}\ni g_{i+1, j+1}\qquad \text{and} \qquad g_{i,j}\cdot y_{j+1}\ni g_{i, j+1},
    \]
    which is the required point satisfying the conditions of the inductive step, thereby completing the induction (see Figure~\ref{fig:staircase}).

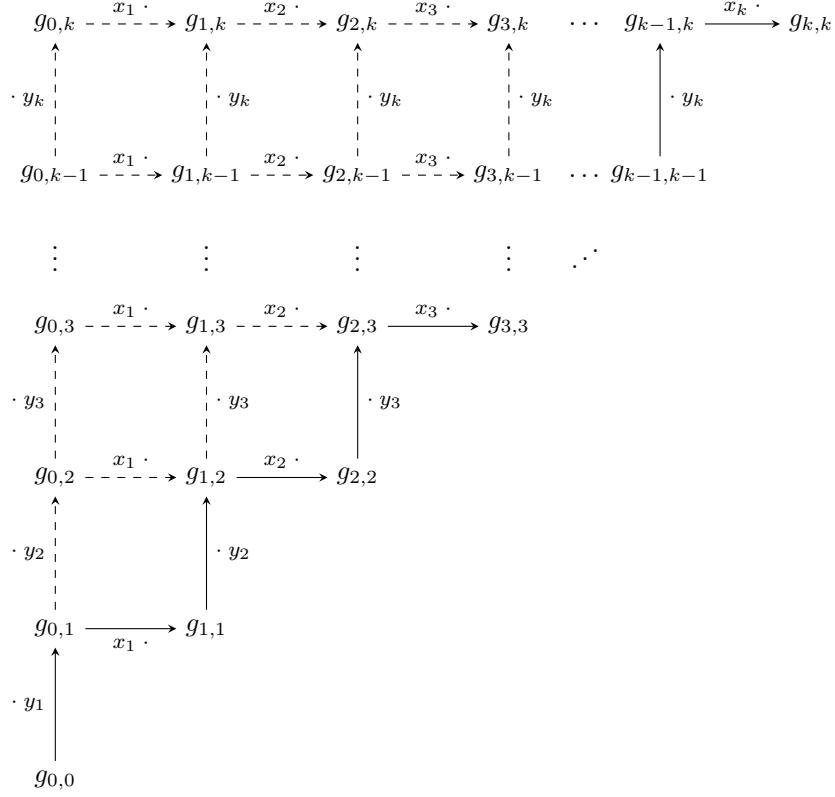
\begin{figure}[H]
\begin{center}
\begin{tikzpicture}[>=stealth, node distance=2cm]
  
  \node (a2) at (6,0) {$g_{0,0}$};
  \node (c2) at (8,2) {$g_{1,1}$};
  \node (c3) at (10,4) {$g_{2,2}$};     \node at (13,7) {$\iddots$};
  \node (g23) at (10,6) {$g_{2,3}$};    \node at (10,7) {$\vdots $};
  \node (g33) at (12,6) {$g_{3,3}$};    \node at (12,7) {$\vdots $};
  \node (k-1) at (14,8) {$g_{k-1,k-1}$};
  \node (k-1v) at (14,10) {$g_{k-1,k}$};
  \node (k) at (16,10) {$g_{k,k}$};
  
  \node (d2) at (6,2) {$g_{0,1}$};
  \node (d3) at (8,4) {$g_{1,2}$};      
  \node (d4) at (6,4) {$g_{0,2}$};      
  \node (g03) at (6,6) {$g_{0,3}$};     \node at (6,7) {$\vdots $};
  \node (g13) at (8,6) {$g_{1,3}$};     \node at (8,7) {$\vdots $};
  
  \node (g0k-1) at (6,8) {$g_{0,k-1}$};
  \node (g0k) at (6,10) {$g_{0,k}$};
  \node (g1k-1) at (8,8) {$g_{1,k-1}$};
  \node (g1k) at (8,10) {$g_{1,k}$};
  \node (g2k-1) at (10,8) {$g_{2,k-1}$};
  \node (g2k) at (10,10) {$g_{2,k}$};
  \node (g3k-1) at (12,8) {$g_{3,k-1}$};    \node at (13,8) {$\hdots $};
  \node (g3k) at (12,10) {$g_{3,k}$};       \node at (13,10) {$\hdots $};


  \draw[->] (a2) -- node[midway,left]{ \footnotesize $\mathrel{\cdot} y_1$ } (d2);
  \draw[->] (d2) -- node[midway,below]{\footnotesize $x_1\mathrel{\cdot}$} (c2);
    \draw[->] (d3) -- node[midway,above]{\footnotesize $x_2\mathrel{\cdot}$} (c3);
    \draw[->] (c2) -- node[midway,right]{ \footnotesize $\mathrel{\cdot} y_2$ } (d3);
    \draw[->] (k-1) -- node[midway,right]{\footnotesize $\mathrel{\cdot} y_k$} (k-1v);
  \draw[->] (k-1v) -- node[midway,above]{\footnotesize $x_k\mathrel{\cdot}$} (k);
  \draw[->] (c3) -- node[midway,right]{\footnotesize $\mathrel{\cdot} y_3$} (g23);
  \draw[->] (g23) -- node[midway,above]{\footnotesize $x_3\mathrel{\cdot}$} (g33);
    \draw[->, dashed] (d4) -- node[midway,above]{\footnotesize $x_1\mathrel{\cdot}$} (d3);
    \draw[->, dashed] (d2) -- node[midway,left]{ \footnotesize $\mathrel{\cdot} y_2$ } (d4);
    \draw[->, dashed] (g03) -- node[midway,above]{\footnotesize $x_1\mathrel{\cdot}$} (g13);
    \draw[->, dashed] (d4) -- node[midway,left]{ \footnotesize $\mathrel{\cdot} y_3$ } (g03);
    \draw[->, dashed] (g13) -- node[midway,above]{\footnotesize $x_2\mathrel{\cdot}$} (g23);
    \draw[->, dashed] (d3) -- node[midway,right]{ \footnotesize $\mathrel{\cdot} y_3$ } (g13);

    \draw[->, dashed] (g0k-1) -- node[midway,left]{\footnotesize $\mathrel{\cdot} y_k$} (g0k);
    \draw[->, dashed] (g1k-1) -- node[midway,right]{\footnotesize $\mathrel{\cdot} y_k$} (g1k);
    \draw[->, dashed] (g2k-1) -- node[midway,right]{\footnotesize $\mathrel{\cdot} y_k$} (g2k);
    \draw[->, dashed] (g3k-1) -- node[midway,right]{\footnotesize $\mathrel{\cdot} y_k$} (g3k);
    \draw[->, dashed] (g0k-1) -- node[midway,above]{\footnotesize $x_1\mathrel{\cdot}$} (g1k-1);
    \draw[->, dashed] (g0k) -- node[midway,above]{\footnotesize $x_1\mathrel{\cdot}$} (g1k);
    \draw[->, dashed] (g1k-1) -- node[midway,above]{\footnotesize $x_2\mathrel{\cdot}$} (g2k-1);
    \draw[->, dashed] (g1k) -- node[midway,above]{\footnotesize $x_2\mathrel{\cdot}$} (g2k);
    \draw[->, dashed] (g2k-1) -- node[midway,above]{\footnotesize $x_3\mathrel{\cdot}$} (g3k-1);
    \draw[->, dashed] (g2k) -- node[midway,above]{\footnotesize $x_3\mathrel{\cdot}$} (g3k);
    
\end{tikzpicture}
\caption{Generating the octant $\mathbb{O}(k)$ from the upper staircase through $\text{(p)}  \qquad (a\cdot b)\cdot c\subseteq a\cdot (b\cdot c)$.}
  \label{fig:staircase}
\end{center}
\end{figure}

    \textbf{Part IV.} Having identified the octant within the refuting model, we define a tiling $\tau:\mathbb{O}(k)\to \mathcal{W}$ by 
    \[
        (m,n)\mapsto t\in \mathcal{W} \text{ such that } g_{m,n}\Vdash \mathtt{t}.
    \]    
    We show that $\tau$ is well-defined and satisfies the tiling conditions. 
    
    For well-definedness, we first verify that each octant point $g_{m,n}$ satisfies \textit{at most} one tile formula $\mathtt{t}$ for $t\in \mathcal{W}$. For convenience, we restrict to points with $m<n$. This suffices because tiling $\{(m,n)\mid 0\leq m<n\leq k\}$ for all $k$ is equivalent to tiling $\{(m,n)\mid 0\leq m\leq n\leq k\}$ for all $k$. 
    
    At the same time, we verify that each octant point satisfies at most one $\mathtt{m_{i,j}}$ for $i,j\in \{0,1\}$; cf.~\eqref{eq:diagonal1}, this is already verified for the diagonal points $g_{m,m}$, $1\leq m\leq k$. Therefore, it is enough to show the following:
    \begin{center}
        If $g_{m+1,n}$ satisfies at most one $\mathtt{m_{i,j}}$, \\ 
        then $g_{m,n}$ satisfies at most one $\mathtt{m_{i,j}}$ and at most one $\mathtt{t}$ for $t\in \mathcal{W}$.
    \end{center}
    So suppose $g_{m+1,n}$ satisfies at most one $\mathtt{m_{i,j}}$. Since $x_{m+1}\Vdash \mathtt{x'}$, we have
    \[
        x_{m+1}\Vdash \bigwedge_{(i,j)\neq (i',j')}\left[(\mathtt{m_{i,j}}\land \mathtt{m_{i',j'}})\to (\mathtt{m_{0,0}}\land \mathtt{m_{0,1}})\right]\land \bigwedge_{\substack{t,t'\in \mathcal{W}, \\ t\neq t'}}\left[(\mathtt{t}\land \mathtt{t}')\to (\mathtt{m_{0,0}}\land \mathtt{m_{0,1}})\right].
    \]
    From $x_{m+1}\cdot g_{m,n}\ni g_{m+1, n}$, it then follows by modus tollens that $g_{m,n}$ satisfies at most one $\mathtt{m_{i,j}}$ and at most one $\mathtt{t}$, as required.

    Next, we show that for all $0\leq m\leq n\leq k$,
    \begin{align}
        g_{m,n}\Vdash \bigvee_{i,j\in\{0,1\}}\mathtt{G_{i,j}}.\label{eq:octantisgrid}
    \end{align}
For $m=n=0$, we are done, since $g_{0,0}\Vdash \mathtt{G_{0,0}}$. So suppose $(m,n)\neq (0,0)$. By~\eqref{eq:stair}, if $m>0$, then
\[
    ((\cdots((((g_{0,0}\cdot x_1)\cdot y_1)\cdot x_2)\cdot y_2)\cdots )\cdot x_m)\cdot y_m\ni g_{m,m},
\]
and hence, by~\eqref{eq:xOntheleft}, we have (also for $m=0$)
\[
    ((\cdots (((((\cdots((((g_{0,0}\cdot x_1)\cdot y_1)\cdot x_2)\cdot y_2)\cdots )\cdot x_m)\cdot y_m)\cdot y_{m+1})\cdot y_{m+2})\cdots )\cdot y_{n-1})\cdot y_{n}\ni g_{m,n}.
\]
Repeated applications of (p) then yield
\[
    g_{0,0}\cdot (x_1\cdot (y_1\cdot (x_2\cdot (y_2\cdot (\cdots(x_m\cdot( y_m\cdot (y_{m+1}\cdot (y_{m+2}\cdot (\cdots (y_{n-1}\cdot y_{n})\cdots)))))\cdots)))))\ni g_{m,n}.
\]
Thus there exists a point $z$ such that $g_{0,0}\cdot z\ni g_{m,n}$ and
\[
    x_1\cdot (y_1\cdot (x_2\cdot (y_2\cdot (\cdots(x_m\cdot( y_m\cdot (y_{m+1}\cdot (y_{m+2}\cdot (\cdots (y_{n-1}\cdot y_{n})\cdots)))))\cdots))))\ni z.
\]
Since all $x_l\Vdash \mathtt{x'}$ and all $y_l\Vdash \mathtt{y'}$, we have $x_l,y_l\Vdash \mathtt{p_\top}\land (\mathtt{p_\top}\to \mathtt{p_\top})$, whence $z\Vdash \mathtt{p_\top}$ by induction. Hence, using that $g_{0,0}\Vdash \gamma$, we precisely get that 
    \[
        g_{m,n}\Vdash \bigvee_{i,j\in\{0,1\}}\mathtt{G_{i,j}}.
    \]

In particular, each octant point $g_{m,n}$ satisfies \textit{at least} one tile formula $\mathtt{t}$, and therefore exactly one. Thus $\tau$ is well-defined. 

It remains to verify that $\tau$ meets the tiling conditions: 
\[
    \tau_E(i,j) = \tau_W(i+1,j) \quad \text{and} \quad \tau_N(i,j) = \tau_S(i,j+1),
\]
which, by the definition of $\tau$ and the formulas $\mathtt{t}, \mathtt{R}(t)$, $\mathtt{U}(t)$, amounts to showing that
\[
    \text{if }g_{i,j}\Vdash \mathtt{t} \text{ then } g_{i+1,j}\Vdash \mathtt{R}(t), \qquad \text{and} \qquad 
    \text{if }g_{i,j}\Vdash \mathtt{t} \text{ then } g_{i,j+1}\Vdash \mathtt{U}(t).
\]
To this end, we refine~\eqref{eq:octantisgrid} by showing that for all $0\leq i\leq j\leq k$, 
 \begin{align*}
     &\text{If $i$ and $j$ are both even: } && g_{i,j}\Vdash \mathtt{G_{0,0}},\\
    &\text{If $i$ is even and $j$ is odd: } && g_{i, j} \Vdash \mathtt{G_{0,1}},\\
    &\text{If $i$ is odd and $j$ is even: } && g_{i,j} \Vdash \mathtt{G_{1,0}},\\
    &\text{If $i$ and $j$ are both odd: } && g_{i,j} \Vdash \mathtt{G_{1,1}}.
 \end{align*}
The proof is by induction, with the diagonal points $g_{i,i}$ for $0\leq i\leq k$ as the base case. For $i=0$, we already have $g_{0,0}\Vdash \mathtt{G_{0,0}}$. For $i>0$, recall (cf.~\eqref{eq:diagonal1}) that 
    \[
        \text{for $i$ even: }\quad g_{i,i}\nVdash \mathtt{m_{0,0}^c}, \qquad \text{for $i$ odd: }\quad g_{i,i}\nVdash \mathtt{m_{1,1}^c}. 
    \]
A fortiori, 
    \[
        \text{for $i$ even: }\quad g_{i,i}\nVdash \mathtt{G_{0,1}}\lor \mathtt{G_{1,0}}\lor \mathtt{G_{1,1}},\qquad \text{for $i$ odd: }\quad g_{i,i}\nVdash \mathtt{G_{0,1}}\lor \mathtt{G_{1,0}}\lor \mathtt{G_{0,0}}, 
    \]
hence
    \[
        \text{for $i$ even: }\quad g_{i,i}\Vdash \mathtt{G_{0,0}}, \qquad \text{for $i$ odd: }\quad g_{i,i}\Vdash \mathtt{G_{1,1}}. 
    \]
This completes the base case. In the inductive step, for $0\leq i\leq j<k$, we show:
\begin{center}
    If $(i,j)$ and $(i+1,j+1)$ satisfy the inductive claim, then so does $(i,j+1)$.
\end{center}
As part of the inductive step, it is expedient to verify the tiling conditions as well, thereby completing the proof of the lemma.

This divides into four cases, depending on the parity of $i$ and $j$. As all cases are analogous, we consider the case where both $i$ and $j$ are even. Then, by assumption, $g_{i,j}\Vdash \mathtt{G_{0,0}}$ and $g_{i+1,j+1}\Vdash \mathtt{G_{1,1}}$. Thus, for some $t\in \mathcal{W}$, 
\[
    g_{i,j}\Vdash \mathtt{m_{0,0}}\land \mathtt{t}\land [\mathtt{y}\to(\mathtt{m_{0,0}}\lor (\mathtt{m_{0,1}} \land \mathtt{U}(t)))]
\]
 We are to show that for some $t'\in \mathcal{W}$,
 \[
    g_{i,j+1}\Vdash \mathtt{G_{0,1}}\land \mathtt{U}(t)\land \mathtt{t'} \qquad \text{and}\qquad  g_{i+1,j+1}\Vdash \mathtt{R}(t').
 \]
Since $j+1>i$, the point $g_{i,j+1}$ satisfies at most one $\mathtt{m_{i,j}}$ and exactly one $\mathtt{t'}$. Consequently, as we know that $g_{i,j+1}\Vdash \bigvee_{m,n\in\{0,1\}}\mathtt{G_{m,n}}$, to deduce $g_{i,j+1}\Vdash \mathtt{G_{0,1}}$, it suffices to show that $g_{i,j+1}\Vdash \mathtt{m_{0,1}}$. 

From $g_{i,j}\Vdash \mathtt{y}\to(\mathtt{m_{0,0}}\lor (\mathtt{m_{0,1}} \land \mathtt{U}(t)))$, $y_{j+1}\Vdash \mathtt{y}$, and $g_{i,j}\cdot y_{j+1}\ni g_{i,j+1}$ (cf.~\eqref{eq:xOntheleft}), we get
\[
    g_{i,j+1}\Vdash \mathtt{m_{0,0}}\lor (\mathtt{m_{0,1}} \land \mathtt{U}(t)).
\]
If we had $g_{i,j+1}\Vdash \mathtt{m_{0,0}}$, then $g_{i,j+1}\Vdash \mathtt{m_{0,0}}\land \mathtt{t'}$. Hence, since $x_{i+1}\cdot g_{i,j+1}\ni g_{i+1,j+1}$ and $x_{i+1}\Vdash \mathtt{x'}$ -- which implies $x_{i+1}\Vdash (\mathtt{m_{0,0}}\land \mathtt{t'})\to (\mathtt{m_{0,0}}\lor (\mathtt{m_{1,0}}\land \mathtt{R}(t')))$ -- it would follow that
\[
    g_{i+1, j+1}\Vdash \mathtt{m_{0,0}}\lor \mathtt{m_{1,0}}.
\]
But that contradicts $g_{i+1, j+1}\Vdash\mathtt{G_{1,1}}$, as this implies $g_{i+1, j+1}\Vdash\mathtt{m_{1,1}}$ and each octant point satisfies at most one of $\mathtt{m_{0,0}}, \hdots ,\mathtt{m_{1,1}}$. Thus $g_{i,j+1}\Vdash \mathtt{m_{0,1}} \land \mathtt{U}(t)$, and therefore
\[
    g_{i,j+1}\Vdash \mathtt{G_{0,1}} \land \mathtt{U}(t).
\]
By the same argument, because $x_{i+1}\cdot g_{i,j+1}\ni g_{i+1,j+1}$ and $x_{i+1}\Vdash \mathtt{x'}$ -- which also implies $x_{i+1}\Vdash (\mathtt{m_{0,1}}\land \mathtt{t'})\to (\mathtt{m_{0,1}}\lor (\mathtt{m_{1,1}}\land \mathtt{R}(t')))$ -- we obtain
\[
    g_{i+1, j+1}\Vdash \mathtt{m_{0,1}}\lor (\mathtt{m_{1,1}}\land \mathtt{R}(t')).
\]
Since $g_{i+1, j+1}\Vdash\mathtt{m_{1,1}}$, we have $g_{i+1, j+1}\nVdash \mathtt{m_{0,1}}$, whence
\[
    g_{i+1, j+1}\Vdash \mathtt{R}(t'),
\]
as required. This concludes the proof of the lemma.
\end{proof}

\subsection{The second tiling lemma.} The proof of the second tiling lemma requires a definition and a couple of auxiliary results. We begin with the definition.

\begin{Definition}
    Let $(K, \cdot)$ and $(K', \cdot')$ be pairs such that $\cdot:K\times K\to \mathcal{P}(K)$ and $\cdot':K'\times K'\to \mathcal{P}(K')$. A map $f:(K, \cdot)\to (K', \cdot')$ is a \textit{p-morphism} if 
    \begin{enumerate}[leftmargin=50pt]
        \item [\textnormal{(forth)}] if $x\cdot y\ni z$, then $f(x)\cdot'f(y)\ni f(z)$,
        \item [\textnormal{(back)}] if $f(x)\cdot'y'\ni z'$, then there are $y,z\in K$ such that $f(y)=y', f(z)=z'$ and $x\cdot y\ni z$.\footnote{Typically, for Routley--Meyer frames $(K, \cdot, N)$, $(K', \cdot', N')$, one would include conditions on $f$ pertaining to $N$, $N'$ as well. However, for our purposes, the current definition will do.}
    \end{enumerate}
\end{Definition}
On such pairs $(K, \cdot)$, given any valuation $V:\mathsf{Prop}\to \mathcal{P}(K)$, the clauses of Definition~\ref{def:clauses} are well defined. The following is then a standard result.

\begin{Fact}\label{fact:pmorph}
     Let $f:(K,\cdot)\to (K', \cdot')$ be a p-morphism. For any valuation $V':\mathsf{Prop}\to \mathcal{P}(K')$, setting $V(p)\mathrel{:=}f^{-1}[V'(p)]$ defines a structure $(K, \cdot, V)$ such that for all $x\in K$ and $\varphi\in\mathcal{L}$,
     \[
        (K, \cdot, V), x \Vdash \varphi \qquad \text{iff}\qquad (K', \cdot', V'), f(x) \Vdash \varphi.
     \]
\end{Fact}
We now define a p-morphism that connects the semilattice frame $(\mathcal{P}_{\mathrm{fin}}(\mathbb{N}), \cup, \varnothing)$ (recall Example~\ref{ex:finitepower}) to a structure closely related to tilings of the quadrant $\mathbb{N}^2$.\footnote{I thank Johan van Benthem for repeatedly encouraging me to seek a direct connection.}

\begin{Lemma}
    Let $(\mathcal{P}_{\mathrm{fin}}(\mathbb{N}), \cup)$ be the semilattice of finite subsets of $\mathbb{N}$ under union, and let $(\mathbb{N}^2, \cdot)$ be the structure where 
    \[
        (m,n)\cdot (m',n')\ni (m'',n'')\quad \text{iff}\quad \begin{cases}
            \max\{m,m'\} \leq m'' \leq m+m',\text{ and}\\
            \max\{n,n'\} \leq n'' \leq n+n'.
            \end{cases}
    \]
    Let $2\mathbb{N}=\{2n\mid n\in \mathbb{N}\}$ denote the even numbers and $2\mathbb{N}+1=\{2n+1\mid n\in \mathbb{N}\}$ the odd numbers. Then $f:(\mathcal{P}_{\mathrm{fin}}(\mathbb{N}), \cup)\to (\mathbb{N}^2, \cdot)$ given by
    \[
        f(X)\mathrel{:=} (|X\cap 2\mathbb{N}|, |X\cap (2\mathbb{N}+1)|)
    \]
    is a p-morphism.
\end{Lemma}
\begin{proof}
    For (forth), simply note that if $X\cup Y=Z$, then $Z$ contains all even (resp.\ odd) numbers of $X$ and $Y$, and, conversely, $X$ and $Y$ jointly contain all even (resp.\ odd) numbers of $Z$.

    For (back), assume $(|X\cap 2\mathbb{N}|, |X\cap (2\mathbb{N}+1)|)\cdot (m', n')\ni (m'', n'')$, i.e.,
    \begin{align*}
        \max\{|X\cap 2\mathbb{N}|,m'\} &\leq m'' \leq |X\cap 2\mathbb{N}|+m',\\
        \max\{|X\cap (2\mathbb{N}+1)|,n'\} &\leq n'' \leq |X\cap (2\mathbb{N}+1)|+n'.
    \end{align*}
    Let $Y\in \mathcal{P}_{\mathrm{fin}}(\mathbb{N})$ be so that $Y\cap 2\mathbb{N}$ contains $m''-|X\cap 2\mathbb{N}|$ even numbers not in $X$ and $m'-(m''-|X\cap 2\mathbb{N}|)$ even numbers from $X$. This is possible since $X$ is finite, $m''-|X\cap 2\mathbb{N}|\geq 0$, and $m'\geq m''-|X\cap 2\mathbb{N}|$. The odd elements of $Y$ are defined analogously. Set $Z\mathrel{:=}X\cup Y$. We claim that $f(Y)=(m', n')$ and $f(Z)=(m'', n'')$. By construction, 
    \[
        |Y\cap 2\mathbb{N}|=(m''-|X\cap 2\mathbb{N}|)+(m'-(m''-|X\cap 2\mathbb{N}|))=m',
    \]
    and similarly $|Y\cap (2\mathbb{N}+1)|=n'$. Hence $f(Y)=(m', n')$. Moreover, 
    \[
        |Z\cap2\mathbb{N}|=|(X\cup Y)\cap2\mathbb{N}|=|X\cap 2\mathbb{N}|+(m''-|X\cap 2\mathbb{N}|)=m'',
    \]
    and similarly $|Z\cap(2\mathbb{N}+1)|=n''$. Therefore $f(Z)=(m'',n'')$, as claimed.
\end{proof}
With this in hand, we can prove the second tiling lemma.

\begin{Lemma}\label{lm:tilingpowerset}
    If $\mathcal{W}$ tiles $\mathbb{Z}^2$, then $(\mathcal{P}_{\mathrm{fin}}(\mathbb{N}), \cup, \varnothing)\nvDash \psi^{}_\mathcal{W}$.
\end{Lemma}
\begin{proof}
     Let $\mathcal{W}$ be an arbitrary finite tile set, and suppose $\mathcal{W}$ tiles $\mathbb{Z}^2$. Then $\mathcal{W}$ tiles the quadrant; let $\tau:\mathbb{N}^2\to \mathcal{W}$ be such a tiling. By the preceding lemma, since $f(\varnothing)=(0,0)$, it suffices to define a valuation $V$ on $(\mathbb{N}^2, \cdot)$ such that $(\mathbb{N}^2, \cdot, V), (0,0)\nVdash \psi^{}_\mathcal{W}$. We define $V$ as follows.    
\begin{align*}
        V(\mathtt{m_{0,0}})&\mathrel{:=}\big\{(m,n)\mid m\text{ is even and }n \text{ is even} \big\},\\
        V(\mathtt{m_{0,1}})&\mathrel{:=}\big\{(m,n)\mid m \text{ is even and }n \text{ is odd}  \big\},\\
        V(\mathtt{m_{1,0}})&\mathrel{:=}\big\{(m,n)\mid m \text{ is odd and }n \text{ is even}  \big\},\\
        V(\mathtt{m_{1,1}})&\mathrel{:=}\big\{(m,n)\mid m \text{ is odd and }n \text{ is odd}  \big\},\\
        V(\mathtt{x})&\mathrel{:=}\big\{(1,0)\big\},\\
        V(\mathtt{y})&\mathrel{:=}\big\{(0,1) \big\},\\
        V(\mathtt{p_\top})&\mathrel{:=}\mathbb{N}^2,\\
        V(\mathtt{t})&\mathrel{:=}\{(m,n)\mid \tau(m,n)=t\}.
\end{align*}
 Since $(0,0)\cdot (0,0)\ni (0,0)$, to prove that $(\mathbb{N}^2, \cdot, V), (0,0)\nVdash \psi^{}_\mathcal{W}$, it is enough to verify that
\[
    (0,0)\Vdash \alpha \land\beta_1\land \beta_2\land  \gamma\land \mathtt{G_{0,0}} \qquad \text{and} \qquad (0,0)\nVdash \mathtt{x'}\to\mathtt{m_{1,0}^c}.
\] 
To this end, we first show that
\begin{align}
    (1,0)\Vdash \mathtt{x'} \qquad \text{and}\qquad (0,1)\Vdash \mathtt{y'}.\label{eq:evenandodd}
\end{align}
By definition, $(0,1)\Vdash \mathtt{y}$. Further, as all $X\in \mathcal{P}_{\mathrm{fin}}(\mathbb{N})$ satisfy $\mathtt{p_\top}$, we have $(0,1)\Vdash \mathtt{p_\top}\land (\mathtt{p_\top}\to \mathtt{p_\top})$, and so $(0,1)\Vdash \mathtt{y'}$. Likewise, $(1,0)\Vdash \mathtt{x} \land \mathtt{p_\top}\land (\mathtt{p_\top}\to \mathtt{p_\top})$. 

Next, observe that $V(\mathtt{m_{i,j}})$ and $V(\mathtt{m_{i',j'}})$ are disjoint for $(i,j)\neq (i', j')$, hence vacuously
\[
    (1,0)\Vdash \bigwedge_{(i,j)\neq (i',j')}\left[(\mathtt{m_{i,j}}\land \mathtt{m_{i',j'}})\to (\mathtt{m_{0,0}}\land \mathtt{m_{0,1}})\right].
\]
Similarly, since $\tau$ is a function, $V(\mathtt{t})$ and $V(\mathtt{t'})$ are disjoint for $t\neq t'$, and hence
\[
    (1,0)\Vdash \bigwedge_{t,t'\in \mathcal{W},t\neq t'}\left[(\mathtt{t}\land \mathtt{t}')\to (\mathtt{m_{0,0}}\land \mathtt{m_{0,1}})\right].
\]
Thus, to conclude $(1,0)\Vdash \mathtt{x'}$, we need only show that
\[
    (1,0)\Vdash \bigwedge_{t\in \mathcal{W}, i,j\in \{0,1\}}\left[(\mathtt{m_{i,j}}\land \mathtt{t})\to (\mathtt{m_{i,j}}\lor (\mathtt{m_{1-i,j}}\land \mathtt{R}(t)))\right].
\]
Accordingly, let $t\in\mathcal{W}$, $i,j\in \{0,1\}$, and $(m,n),(m',n')\in \mathbb{N}^2$ be arbitrary with
\[
    (m,n)\Vdash \mathtt{m_{i,j}}\land \mathtt{t}\qquad \text{and}\qquad (1,0)\cdot (m,n)\ni(m',n').
\]
Then, by definition of $\cdot$,
\[
    \max\{1,m\}\leq m'\leq 1+m\qquad \text{and} \qquad \max\{0,n\}\leq n'\leq 0+n
\]
Hence $n=n'$, and either $m=m'$ or $m+1=m'$. If $m=m'$, then $(m',n')=(m,n)\Vdash \mathtt{m_{i,j}}$, so $(m',n')$ satisfies the first disjunct of the consequent. If $m+1=m'$, then $(m',n')=(m+1,n)$, and $(m,n)\Vdash \mathtt{m_{i,j}}$ entails $(m',n')\Vdash\mathtt{m_{1-i,j}}$. Moreover, since $\tau$ satisfies the tiling conditions, $(m,n)\Vdash \mathtt{t}$ implies $(m',n')\Vdash\mathtt{R}(t)$, so $(m',n')$ satisfies the second disjunct. Thus we have proven~\eqref{eq:evenandodd}. 

Using this, we obtain $(0,0)\nVdash \mathtt{x'}\to\mathtt{m_{1,0}^c}$, since $(1,0)\Vdash \mathtt{x'}$ and $(0,0)\cdot (1,0)\ni(1,0)\nVdash \mathtt{m_{1,0}^c}$.

Further, $(0,0) \Vdash \alpha$ because $(0,0)\cdot (m,n)\ni (m',n')$ only if $(m,n)=(m', n')$. 

By this observation, to show $(0,0)\Vdash \beta_1$, it suffices to show that for every $(m,n)\in \mathbb{N}^2$ and $i\in\{0,1\}$: if $(m,n)\nVdash \mathtt{m_{i,1-i}^c}$, then $(m,n)\nVdash \mathtt{y'}\to \mathtt{m_{i,i}^c}$. So suppose $(m,n)\nVdash \mathtt{m_{i,1-i}^c}$. Then $(m,n)\Vdash \mathtt{m_{i,1-i}}$, hence $(m,n+1)\nVdash \mathtt{m_{i,i}^c}$. But $(0,1)\Vdash \mathtt{y'}$, so $(m,n)\cdot (0,1) \ni (m,n+1)$ witnesses that $(m,n)\nVdash \mathtt{y'}\to \mathtt{m_{i,i}^c}$. Thus $(0,0)\Vdash \beta_1$. A similar argument shows $(0,0)\Vdash \beta_2$.

Lastly, to prove that $(0,0) \Vdash \gamma\land \mathtt{G_{0,0}}$, we show that for any $(m,n)\in\mathbb{N}^2$: if $(m,n)\Vdash \mathtt{m_{i,j}}$, then $(m,n)\Vdash \mathtt{G_{i,j}}$. This suffices since the sets $V(\mathtt{m_{0,0}}),$ $V(\mathtt{m_{0,1}}),$ $V(\mathtt{m_{1,0}}),$ $V(\mathtt{m_{1,1}})$ partition $\mathbb{N}^2$, and $(0,0)\in V(\mathtt{m_{0,0}})$. 

So suppose $(m,n)\Vdash \mathtt{m_{i,j}}$ for some $(m,n)\in\mathbb{N}^2$ and $i,j\in\{0,1\}$. Take $t\in \mathcal{W}$ with $\tau(m,n)=t$. Then $(m,n)\Vdash \mathtt{t}$. It is therefore enough to show that
    \[
        (m,n)\Vdash \mathtt{y}\to (\mathtt{m_{i,j}}\lor (\mathtt{m_{i,1-j}}\land \mathtt{U}(t))).
    \]
Since $(m,n)\cdot (0,1) \ni (m', n')$ only if $(m', n')=(m, n)$ or $(m', n')=(m, n+1)$, there are two cases. If $(m', n')=(m, n)$, then $(m', n')\Vdash \mathtt{m_{i,j}}$. If $(m',n')=(m,n+1)$, then $(m,n)\Vdash \mathtt{m_{i,j}}$ implies $(m,n+1)\Vdash\mathtt{m_{i,1-j}}$, and $(m,n)\Vdash \mathtt{t}$ implies $(m,n+1)\Vdash\mathtt{U}(t)$. Thus, in either case, 
\[
    (m',n')\Vdash \mathtt{m_{i,j}}\lor (\mathtt{m_{i,j-1}}\land \mathtt{U}(t)).
\]
This completes the proof of $(0,0) \Vdash \gamma\land \mathtt{G_{0,0}}$, and hence of the lemma.
\end{proof}

\section{Results}\label{sec:results} The tiling lemmas together imply our main theorem.
\begin{Theorem}\label{th:main}
    For any set of formulas $\mathsf{L}\subseteq \mathcal{L}$, if 
    \[
        \mathsf{TWJ}^+\subseteq \mathsf{L}\subseteq \mathrm{Log}(\mathcal{P}_{\mathrm{fin}}(\mathbb{N}), \cup, \varnothing),
    \]
    then $\mathsf{L}$ is undecidable.
\end{Theorem}
\begin{proof}
    Let $\mathcal{W}$ be a finite tile set. We show that 
    \[
    \psi^{}_{\mathcal{W}}\notin \mathsf{L}\qquad \text{iff}\qquad \text{$\mathcal{W}$ tiles $\mathbb{Z}^2$.}
    \]
    Suppose $\psi^{}_{\mathcal{W}}\notin \mathsf{L}$. Since $\mathsf{TWJ}^+ \subseteq \mathsf{L}$, it follows that $\psi_{\mathcal{W}} \notin \mathsf{TWJ}^+$. By Lemma~\ref{lm:tiling}, $\mathcal{W}$ tiles $\mathbb{O}(k)$ for all $k\in \mathbb{N}$, and hence $\mathcal{W}$ tiles $\mathbb{Z}^2$. 
    
    Conversely, if $\mathcal{W}$ tiles $\mathbb{Z}^2$, then by Lemma~\ref{lm:tilingpowerset}, $\psi^{}_{\mathcal{W}}\notin \mathrm{Log}(\mathcal{P}_{\mathrm{fin}}(\mathbb{N}), \cup, \varnothing)$, so $\psi^{}_{\mathcal{W}}\notin \mathsf{L}$.

    Thus the tiling problem reduces to the decision problem for $\mathsf{L}$.
\end{proof}
We note that in Theorem~\ref{th:main}, $\mathrm{Log}(\mathcal{P}_{\mathrm{fin}}(\mathbb{N}), \cup, \varnothing)$ can be replaced by $\mathrm{Log}(\mathbb{N}^2, \cdot, (0,0))$. The former formulation, however, makes certain consequences more immediate. In particular, since $(\mathcal{P}_{\mathrm{fin}}(\mathbb{N}), \cup, \varnothing)$ is a semilattice, we obtain the following:

\begin{Theorem}
    $\mathsf{S}$ is undecidable.
\end{Theorem}

As a corollary, it follows that $\mathsf{S}$ lacks the finite model property, answering a question recently highlighted in~\cite{Weiss21}.

\begin{Corollary}\label{th:noFMP}
    $\mathsf{S}$ does not have the finite model property.
\end{Corollary}
\begin{proof}
    $\mathsf{S}$ is recursively enumerable, and so is the class of finite semilattices. Therefore, if $\mathsf{S}$ had the finite model property, it would be co-recursively enumerable, and hence decidable.\footnote{For a witnessing formula, see the earlier conference version of this paper~\cite{Knudstorp2024}. There it is shown that the formula $\psi_{\infty}$ is refuted by an infinite semilattice frame, but by no finite semilattice frame, where $\psi_{\infty}\mathrel{:=}(o\land \psi_1\land \psi_2\land \psi_3)\to e,$ for $e,o$ propositional letters and $\psi_1\mathrel{:=}[e\land (e\to e)] \to o,\psi_2\mathrel{:=}[o\land (o\to o)]\to e,\psi_3\mathrel{:=} (e\lor o)\to [(e\lor o)\land (e\lor o\to e\lor o)].$}
\end{proof}
Theorem~\ref{th:main} further furnishes a new, tiling-based proof of the undecidability of $\mathsf{T}, \mathsf{E}, \mathsf{R}$ and of their $\{\to,\land, \lor\}$-fragments, and likewise for $\mathsf{TWJ},\mathsf{C}, \mathsf{SetFr}$.
\begin{Theorem}
    $\mathsf{TWJ}^+, \mathsf{C}^+,  \mathsf{T}^+, \mathsf{E}^+, \mathsf{R}^+, \mathsf{SetFr}$ are all undecidable.
\end{Theorem}
The logics $\mathsf{TWJ}^+, \mathsf{C}^+,  \mathsf{T}^+, \mathsf{E}^+, \mathsf{R}^+, \mathsf{SetFr}$ are, like $\mathsf{S}$, recursively enumerable, and so are their respective classes of finite frames. Hence Theorem~\ref{th:main} also entails that none of these logics has the finite model property.

Beyond the relevant-logical setting, Theorem~\ref{th:main} can also be used to derive new undecidability results within other frameworks through translation. Some initial steps in this direction are taken in~\cite{KnudstorpPosRel}.
\\\\
We end with a different kind of corollary, one that is seldom stated but significant and therefore, I believe, worth the mention.

Within axiomatic theories such as $\mathsf{ZFC}$, we can formalize statements of the form `$\mathsf{TWJ}^+\vdash \varphi$' and `$\mathsf{TWJ}^+\nvdash \varphi$'. We may therefore ask, for a given $\varphi\in\mathcal{L}$, whether the derivability of $\varphi$ in $\mathsf{TWJ}^+$ is settled by $\mathsf{ZFC}$. 
When presented with the axiomatization of $\mathsf{TWJ}^+$ in~\S\ref{subsec:syntax}, one might, at least initially, expect this question to be trivial: surely $\mathsf{ZFC}$, powerful enough to formalize modern mathematics, should settle whether a formula $\varphi\in\mathcal{L}$ is derivable from the simple axioms and rules of $\mathsf{TWJ}^+$. This expectation, however, is far too optimistic.


%
\begin{Corollary}
     Assume that $\mathsf{ZFC}$ is sound for statements of the form `$\mathsf{TWJ}^+\vdash \varphi$', i.e. that
     \[
        \mathsf{ZFC}\vdash \text{`$\mathsf{TWJ}^+\vdash \varphi$'}\qquad \text{implies} \qquad \mathsf{TWJ}^+\vdash \varphi.
     \]
     Then there are infinitely many $\varphi\in \mathcal{L}$ for which the statement `$\mathsf{TWJ}^+\vdash \varphi$' is independent of $\mathsf{ZFC}$.
\end{Corollary}
\begin{proof}
    Since $\mathsf{ZFC}$ is recursively enumerable, we can define a Turing machine that, on input $\varphi$, searches for a $\mathsf{ZFC}$-proof of `$\mathsf{TWJ}^+\vdash \varphi$' or of its negation `$\mathsf{TWJ}^+\nvdash \varphi$'. 
    
    From this machine, for each finite set of formulas $\Phi\subseteq \mathcal{L}$, we can define $2^{|\Phi|}$ Turing machines that, outside $\Phi$, run the preceding Turing machine, but, within $\Phi$, hard-code the $2^{|\Phi|}$ possible YES/NO answer patterns. One of these is correct for inputs $\varphi\in \Phi$. It is also correct for inputs $\varphi\notin\Phi$ \textit{provided that} `$\mathsf{TWJ}^+\vdash \varphi$' is not independent of $\mathsf{ZFC}$. For if there is a $\mathsf{ZFC}$-proof of `$\mathsf{TWJ}^+\vdash \varphi$', then by soundness, $\mathsf{TWJ}^+\vdash \varphi$. And if there is a $\mathsf{ZFC}$-proof of `$\mathsf{TWJ}^+\nvdash \varphi$', then by consistency (implied by soundness), $\mathsf{TWJ}^+\nvdash \varphi$, since otherwise $\mathsf{ZFC}$ would prove `$\mathsf{TWJ}^+ \vdash \varphi$' by formalizing the corresponding finite derivation.

    Thus it follows from the undecidability of $\mathsf{TWJ}^+$ that there are infinitely many  $\varphi\in \mathcal{L}$ for which `$\mathsf{TWJ}^+\vdash \varphi$' is independent of $\mathsf{ZFC}$.
\end{proof}
Under the weaker assumption that $\mathsf{ZFC}$ is consistent, there are still infinitely many formulas $\varphi \in \mathcal{L}$ such that $\mathsf{TWJ}^+ \nvdash \varphi$ but $\mathsf{ZFC}$ does not prove `$\mathsf{TWJ}^+ \nvdash \varphi$' (only that now we can have $\mathsf{ZFC}\vdash \text{`$\mathsf{TWJ}^+\vdash \varphi$'}$). For if not, the complement $\mathcal{L}\setminus \mathsf{TWJ}^+$ would be recursively enumerable modulo a finite set, and hence recursively enumerable, contradicting the fact that $\mathsf{TWJ}^+$ is recursively enumerable and undecidable.

These arguments are not unique to $\mathsf{TWJ}^+$ or to $\mathsf{ZFC}$. We can swap $\mathsf{TWJ}^+$ for any of the undecidable logics considered above, and $\mathsf{ZFC}$ for any recursively enumerable theory extending Robinson arithmetic $\mathsf{Q}$, such as $\mathsf{ZFC}$ + some large cardinals. Undecidability implies independence galore. 

For an example of a specific formula that exhibits the above independence behaviour, 
let $T$ be a Turing machine that searches for a $\mathsf{ZFC}$-proof of contradiction and halts just in case such a proof is found. Let, in turn, $\mathcal{W}_T$ be a Wang tile set obtained following Berger~\cite{Berger} 
such that $T$ halts iff $\mathcal{W}_T$ does not tile the plane. Then the corresponding tiling formula $\psi^{}_{\mathcal{W}_T}$ exhibits the above independence behaviour, as we have
\begin{align*}
    \mathsf{TWJ}^+ \vdash \psi^{}_{\mathcal{W}_T}
    \quad\text{iff}\quad
    \mathcal{W}_T \text{ does not tile the plane}
    \quad\text{iff}\quad
    \mathsf{ZFC} \text{ is inconsistent}.
\end{align*}
The question of whether $\mathsf{TWJ}^+\vdash \psi^{}_{\mathcal{W}_T}$ is left as an exercise for the reader.
\\\\

\subsection*{Acknowledgements} I am grateful to Shawn Standefer for his thoughtful comments and constant encouragement. I also thank an anonymous referee for corrections to the exposition. 

\subsection*{Funding} The author was supported by the Nothing is Logical (NihiL) project (NWO OC 406.21.CTW.023).

\bibliographystyle{asl}
\bibliography{bibliography}

\end{document}